\documentclass[12pt]{amsart}
\newcommand\ignore[1]{}
\input amssym.def
\usepackage[raiselinks=false,colorlinks=true,citecolor=blue,urlcolor=blue,linkcolor=blue,bookmarksopen=true,pdftex]{hyperref}
\usepackage{booktabs,multirow,lscape} % for tables 
\newtheorem{theorem}{Theorem}[section]
\newtheorem{proposition}[theorem]{Proposition}
\newtheorem{lemma}[theorem]{Lemma}
\newtheorem{remark}[theorem]{Remark}

\newcommand{\res}{\textrm{Res}}

\newcommand{\bz}{\mathbb{Z}}

\newcommand{\br}{\mathbb{R}}
\newcommand{\bc}{\mathbb{C}}

\newcommand{\bs}{\mathbb{S}}

\newcommand{\e}{\varepsilon}
\newcommand{\lr}{\longrightarrow}
\renewcommand{\hom}{\textrm{Hom}}

\begin{document}
\baselineskip=15.5pt

\title%[Common $L_0$-types in discrete series]{$L_0$-types common to a
%Borel-de Siebenthal discrete series and its associated
%holomorphic discrete series}  
[Borel-de Siebenthal and holomorphic discrete series]{Borel-de Siebenthal discrete
  series and \\  associated holomorphic discrete series}  
\author{Pampa Paul}
\author{K. N. Raghavan}
\author{Parameswaran Sankaran}
\address{The Institute of Mathematical Sciences, CIT
Campus, Taramani, Chennai 600113, India}
\email{pampa@imsc.res.in, knr@imsc.res.in, sankaran@imsc.res.in}
\dedicatory{Dedicated to the memory of D. -N. Verma} 


\thispagestyle{empty}
\date{}
\begin{abstract}
Let $G_0$ be a simply connected non-compact real simple Lie group with maximal compact subgroup $K_0$. Assume that rank$(G_0)$ = rank$(K_0)$ so that $G_0$ has discrete series 
representations. If $G_0/K_0$ is Hermitian symmetric, one has a relatively simple discrete series of $G_0$, namely the holomorphic discrete series of $G_0$. Now assume that $G_0/K_0$ is not a Hermitian symmetric space. In this 
case, one has the class of Borel-de Siebenthal discrete series of $G_0$ defined in a manner analogous to the holomorphic discrete series. 
We consider a certain circle subgroup of~$K_0$ whose centralizer~$L_0$ is such that $K_0/L_0$
%We consider the centralizer in $K_0$ of a certain circle subgroup of $K_0$. It turns out that $K_0/L_0$ 
is an irreducible compact Hermitian symmetric space. Let $K_0^*$ be the dual of $K_0$ with respect to $L_0$. Then $K_0^*/L_0$ is an irreducible non-compact Hermitian symmetric space dual to $K_0/L_0$.
In this article, to each Borel-de Siebenthal discrete series of $G_0$, we will associate a holomorphic discrete series of $K_0^*$. Then we show the occurrence of infinitely many common $L_0$-types between these two discrete series under certain conditions. 
\end{abstract}
\maketitle

\noindent
\section{Introduction}
Let $G_0$ be a simply connected non-compact real simple Lie group and let $K_0$ be a maximal compact subgroup of $G_0$.  
Let $T_0\subset K_0$ be a maximal torus. Assume that 
rank$(K_0)=$rank$(G_0)$ so that $G_0$ has discrete series representations. Note that $T_0$ is a Cartan subgroup of $G_0$ as well. Also the condition rank$(K_0)$=rank$(G_0)$ implies that $K_0$ is the fixed point set of a Cartan involution of $G_0$. We shall denote by $\frak{g}_0,\frak{k}_0$, and $\frak{t}_0$ the Lie algebras of $G_0,K_0$, and $T_0$ respectively and by $\frak{g}, \frak{k}$, and $\frak{t}$ the complexifications of $\frak{g}_0,\frak{k}_0,$ and $\frak{t}_0$ respectively.  

Let $\Delta$ be the root 
system of $\frak{g}$ with respect to the Cartan subalgebra $\frak{t}$. Let $\Delta^+$ be a Borel-de Siebenthal positive system so that the set of simple roots $\Psi$ has exactly one non-compact 
root $\nu$.    We may write
$\Delta=\cup _{-2\leq i\leq 2}\Delta_i$ where $\alpha\in\Delta$ belongs to $\Delta_i$ precisely when the coefficient $n_\nu(\alpha)$ of $\nu$ in $\alpha$ when expressed as a sum of simple roots is equal to $i$; the set of compact and non-compact roots of $\frak{g}_0$ are $\Delta_0\cup \Delta_2\cup \Delta_{-2}$ and $\Delta_1\cup \Delta_{-1}$ respectively. 

Let $G$ be the simply connected complexification of $G_0$. The inclusion $\frak{g}_0\hookrightarrow \frak{g}$ defines a homomorphism $p: G_0\lr G$. Let $Q\subset G$ be the parabolic subgroup with Lie algebra $\frak{q}=\frak{l} \oplus \frak{u}_{-1}\oplus \frak{u}_{-2}$, where $\frak{u}_i =\sum_{\alpha \in \Delta_i} \frak{g}_\alpha$ ($-2\leq i\leq 2$), $\frak{g}_\alpha$  is the root space for $\alpha \in \Delta$, and $\frak{l}=\frak{t} \oplus \frak{u}_0$. Let $L$ be the Levi subgroup of $Q$; thus $Lie(L)=\frak{l}$. Then $\bar{L}_0:=p(G_0)\cap Q$ is a real form of $L$ and $L_0:=p^{-1}(\bar{L}_0)$ is the centralizer in $K_0$ of a circle subgroup of $T_0$.

Note that $G_0/L_0$ is an open orbit of the complex flag manifold $G/Q$, $K_0/L_0$ is an irreducible Hermitian symmetric space of compact type and $G_0/L_0 \lr G_0/K_0$ is a fibre bundle projection with fibre $K_0/L_0$. 

Our interest is in the situation when $G_0/K_0$ is not a Hermitian symmetric space.   
%Assume that $G_0/K_0$ is {\it not} a Hermitian symmetric space. 
This condition is equivalent to the requirement that the centre of $K_0$ is discrete. 
We want to consider in this situation the Borel-de Siebenthal discrete series of $G_0$, which was the subject of \O rsted and Wolf \cite{ow}.   This is defined analogously to holomorphic discrete series in the case when $G_0/K_0$ is a Hermitian symmetric space,  and so we first recall that definition.

If $G_0/K_0$ is a Hermitian symmetric space, then $\Delta_2$ and $\Delta_{-2}$ are empty, 
%one has a partition $\Delta=\cup _{-1\leq i\leq 1}\Delta_i$ where $\alpha\in\Delta$ belongs to $\Delta_i$ precisely when the coefficient $n_\nu(\alpha)$ of $\nu$ in $\alpha$ when expressed as a sum of simple roots is equal to $i$, 
and the set of compact and non-compact roots of $\frak{g}_0$ are $\Delta_0$ and $\Delta_1\cup \Delta_{-1}$ respectively. Note that $L_0=K_0$ in this case.
If $\gamma$ is the highest weight of an irreducible representation of $K_0$ such that $\gamma+\rho_{\frak{g}}$ is negative on $\Delta_1$, then $\gamma+\rho_{\frak{g}}$ is the Harish-Chandra parameter of a holomorphic discrete series $\pi_{\gamma+\rho_{\frak{g}}}$ of $G_0$.  The $K_0$-finite part of $\pi_{\gamma+\rho_{\frak{g}}}$ 
is described as $\oplus_{n\geq 0} E_\gamma\otimes S^n(\frak{u}_{-1})$ where $E_\gamma$ is the irreducible $K_0$-representation with highest weight $\gamma$ and $\frak{u}_{-1}=\oplus_{\alpha\in \Delta_{-1}} \frak{g}_\alpha.$  See \cite{hc1} and also \cite{repka}.

Now, turning to the situation when $G_0/K_0$ is not a Hermitian symmetric space, let
$\gamma$ be the highest
 weight of an irreducible representation $E_\gamma$ of $\bar{L}_0$ such that  $\gamma+\rho_{\frak{g}}$ is negative on $\Delta_1\cup\Delta_2$. Here $\rho_{\frak{g}}$ denotes 
half the sum of positive roots of $\frak{g}$. The Borel-de Siebenthal discrete series $\pi_{\gamma+\rho_{\frak{g}}}$ is the discrete series representation of $G_0$ for which the Harish-Chandra parameter is $\gamma+\rho_{\frak{g}}$.    Let $\mu$ be the highest root in~$\Delta^+$, let $\frak{k}_1^\mathbb{C}$ be the simple ideal of ~$\frak{k}$ containing $\frak{g}_\mu$, let
$\frak{k}_1$ be the compact real form of~$\frak{k}_1^\mathbb{C}$ contained in~$\frak{k}_0$, and let $K_1$ be 
the simple factor of~$K_0$ with Lie algebra $\frak{k}_1$.  
The $K_0$-finite part of $\pi_{\gamma+\rho_{\frak{g}}}$ is in fact 
$K_1$-admissible. 
%, which means that, when restricted to~$K_1$,  it breaks up as a direct sum of irreducibles each having finite multiplicity. 
This is a consequence a more general theorem on admissible restrictions due to Kobayashi \cite[Theorem 2.9]{kobayashi94}.   
\O rsted and Wolf \cite{ow} observe this using the description of the $K_0$-finite part of $\pi_{\gamma+\rho_{\frak{g}}}$ in terms of the Dolbeault cohomology as 
%~\cite[Theorem~2.8]{kobannals};  see \cite{ow} for an elementary proof. 
%The $K_0$-finite part can be described---see~\cite{ow}---in terms of the Dolbeault cohomology as 
$\oplus_{m\geq 0}H^s(K_0/L_0;\mathbb{E}_\gamma\otimes \mathbb{S}^m(\frak{u}_{-1}))$ where $s=\dim _\bc K_0/L_0$, $\mathbb{E}_\gamma$ and $\mathbb{S}^m(\frak{u}_{-1})$ denote the holomorphic vector bundles associated to the irreducible $L_0$-module $E_\gamma$ and the $m$-th symmetric power $S^m(\frak{u}_{-1})$ of the irreducible $L_0$-module $\frak{u}_{-1}$ respectively.  

Before proceeding further, we pause to recall here the important notion of admissibility of a representation.  Suppose that $H$ is a Lie group and that $(\pi, V_\pi)$ is a unitary representation of $H$ on a complex Hilbert space $V_\pi$.   Following 
Kobayashi \cite{kobayashi94}, we say that $\pi$ is {\it admissible} if $V_\pi$ is expressible as a Hilbert space direct sum 
$V_\pi=\hat{\oplus}_{\tau\in \hat{H}}m_\tau V_\tau$ where the sum is over the set $\hat{H}$ of all isomorphism classes of irreducible unitary representations 
$(\tau, V_\tau)$ of $H$ and $m_\tau=\dim_\bc(\hom_H(V_\tau, V_\pi))$, the multiplicity of $\tau$ in $\pi$,  is {\it finite} for all $\tau\in \hat{H}$.   If $H_1$ is a closed subgroup of $H$, we say that $(\pi, V_\pi)$ is $H_1$-{\it admissible} if the 
restriction $(\pi|_{H_1},V_\pi)$ is admissible as an $H_1$-representation. 
 
 We regard any $\bar{L}_0$ representation 
as an $L_0$-representation via the covering projection $p|_{L_0}$.  Any $L_0$-representation we consider in this paper arises from 
an $\bar{L}_0$-representation and so we shall abuse notation and simply 
write $L_0$ for $\bar{L}_0$ as well.  

R. Parthasarathy \cite{parthasarathy} obtained essentially 
the same description as above in a more general context 
that includes holomorphic and Borel-de Siebenthal 
discrete series as well as certain limits of discrete series 
representations. We give a brief description of his 
results in Appendix~2 (\S\ref{s:appendixrp}).

 Let $\Delta^\pm_0=\Delta^\pm\cap \Delta_0$.  Then $\Delta^+=\Delta^+_0\cup \Delta_1\cup \Delta_2$. 
The root system of $\frak{k}$ is $\Delta_{\frak{k}}=\Delta_0\cup \Delta_2\cup \Delta_{-2}$, 
and the induced positive system of $\Delta_{\frak{k}}$ is obtained as 
$\Delta^+_{\frak k}=\Delta_0^+\cup\Delta_2$. 
 
Let $(K_0^*,L_0)$ denote the Hermitian symmetric pair dual to the pair $(K_0,L_0)$.  
The set of non-compact roots in $\Delta_{\frak{k}}^+$ equals $\Delta_2$ with respect to the real form $Lie(K_0^*)$ of $\frak{k}$. If $\gamma+\rho_{\frak{g}}$ is the Harish-Chandra parameter of a Borel-de Siebenthal discrete series $\pi_{\gamma+\rho_{\frak{g}}}$ of $G_0$, then the same parameter $\gamma$ determines a holomorphic discrete series of $K_0^*$ with Harish-Chandra parameter $\gamma+\rho_{\frak{k}}$, denoted $\pi_{\gamma+\rho_{\frak{k}}}$. See \S 4. It is a natural question to ask which 
$L_0$-types are common to the Borel-de Siebenthal discrete series $\pi_{\gamma+\rho_{\frak{g}}}$ and the corresponding holomorphic discrete series 
$\pi_{\gamma+\rho_{\frak{k}}}$. 

We shall answer this 
question completely when $\frak{k}_1 \cong \frak{su}(2)$, the so-called quaternionic case. See Theorem \ref{main}. In the non-quaternionic case, we obtain complete  
results assuming that 
(i) the longest element of the Weyl group of $K_0$ preserves 
$\Delta_0$, that is, $K_0^*/L_0$ is of tube type,   and
(ii) there exists a non-trivial one dimensional 
$L_0$-subrepresentation in the symmetric algebra $S^*(\frak{u}_{-1})$.
See Theorem \ref{nonquaternionic} below.
The only Hermitian symmetric spaces that occur as~$K_0^*/L_0$ in our context and are of tube type are: $SO^*(4m)/U(2m)$, $SO_0(2,2m)/SO(2)\times SO(2m)$, $Sp(m,\mathbb{R})/U(m)$.

Note that condition~(i) is  
trivially satisfied in the quaternionic case.   
 The existence of 
non-trivial one-dimensional $L_0$-submodule in the symmetric algebra $S^*(\frak{u}_{-1})$ greatly simplifies the task of detecting occurrence of common $L_0$-types. The classification of 
Borel-de Siebenthal positive systems for which 
such one dimensional exist has been carried out 
by \O rsted and Wolf \cite[\S4]{ow}.   

We now state the main results of this paper.

\begin{theorem}\label{main}
We keep the above notations.   Suppose that $Lie(K_1)\cong \frak{su}(2)$. If $\frak{g}_0=\frak{so}(4,1)$ or $\frak{sp}(1,l-1), l>1$, then 
there
are at most finitely many $L_0$-types  common to $\pi_{\gamma+\rho_{\frak{g}}}$ and $\pi_{\gamma+\rho_{\frak{k}}}$.  Moreover, if $\dim E_\gamma=1$ then there 
are no common $L_0$-types. 

Suppose that $\frak{g}_0\neq \frak{so}(4,1)$ or $\frak{sp}(1,l-1), l>1$. Then each $L_0$-type in the holomorphic discrete series $\pi_{\gamma+\rho_{\frak{k}}}$ occurs in the Borel-de Siebenthal discrete series  $\pi_{\gamma+\rho_{\frak{g}}}$  
with {\em infinite} multiplicity. 
\end{theorem}

The cases $G_0=SO(4,1), Sp(1,l-1)$ 
are exceptional among the quaternionic 
cases in that these are precisely the cases for which  prehomogeneous space $(L,\frak{u}_{1})$ has no (non-constant) relative invariants---equivalently 
$S^m(\frak{u}_{-1}), m\geq 1,$ has no one-dimensional 
$L_0$-subrepresentation.   In the non-quaternionic case, 
we have the following result.

\begin{theorem}\label{nonquaternionic}
With the above notations, suppose that (i) $w_{\frak{k}}^0(\Delta_0)=\Delta_0$ where $w_{\frak{k}}^0$ is the longest 
element of the Weyl group of $K_0$ (equivalently, the Hermitian symmetric space~$K_0^*/L_0$ is of tube type), and, (ii)  
there exists a 
$1$-dimensional $L_0$-submodule in $S^m(\frak{u}_{-1})$ for some $m\geq 1$.  
Then there are infinitely many $L_0$-types common to $\pi_{\gamma+\rho_{\frak{g}}}$, $\pi_{\gamma+\rho_{\frak{k}}}$ and occurring in $\pi_{\gamma+\rho_{\frak{g}}}$ with infinite multiplicity.  Moreover, if $\dim E_\gamma=1$,  then every $L_0$-type occurring in $\pi_{\gamma+\rho_{\frak{k}}}$ occurs
in $\pi_{\gamma+\rho_{\frak{g}}}$ with infinite multiplicity.
\end{theorem}
\noindent

We recall, in Proposition \ref{relativeinvariants}, the 
Borel-de Siebenthal root orders for which condition 
(ii) of the above theorem holds.  We obtain in Proposition  \ref{sumgamma} a criterion for condition (i) to hold. 
For the complete list of non-quarternionic cases in which condition~(i) holds, see Appendix~1~(\S\ref{ss:t:nquart}).
%For the convenience of the reader we tabulate the result in Appendix~1 (\S\ref{s:appendix1}) in the non-quaternionic cases. 
  
The second part of Theorem \ref{main} is a particular case of Theorem \ref{nonquaternionic} (when $Lie(K_1) \cong \frak{su}(2)$, the common $L_0$-types are all in $\pi_{\gamma+\rho_{\frak{k}}}$). The proof of Theorem \ref{main} involves only elementary considerations. But the proof of Theorem \ref{nonquaternionic} involves much deeper results and arguments.

The  existence (or non-existence) of one-dimensional $L_0$-submodules in 
$\oplus_{m\geq 1}S^m(\frak{u}_{-1})$ is closely 
related to the $L_0$-admissibility of  $\pi_{\gamma+\rho_{\frak{g}}}$. 
Note that Theorem \ref{nonquaternionic} implies that, under the condition $w_{\frak{k}}^0(\Delta_0)=\Delta_0$, the restriction 
of the Borel-de Siebenthal discrete series is not 
$L_0$-admissible when $\sum_{m>0}S^m(\frak{u}_{-1})$ has one dimensional subrepresentation. When $\frak{k}_1\cong \frak{su}(2)$ and $\sum_{m>0}S^m(\frak{u}_{-1})$ has no one dimensional submodule, the Borel-de Siebenthal discrete 
series is $L_0$-admissible. In fact, one has the following result:

\begin{proposition}\label{inadmissible}
Suppose that $S^m(\frak{u}_{-1})$ has a one-dimensional $L_0$-subrepresentation for some $m\geq 1$, then the Borel-de Siebenthal discrete series $\pi_{\gamma+\rho_{\frak{g}}}$ 
is not $L_0'$-admissible where $L_0'=[L_0,L_0]$.   
The converse holds if $\frak{k}_1\cong \frak{su}(2)$.   
\end{proposition}
\noindent
For a general criterion for admissibility of restriction to a closed subgroup from a compact Lie group, see~\cite[Theorem~6.3.3]{kobayashipm}.

We also obtain, in Proposition \ref{holadmissibility}, 
a result on the $L_0'$-admissibility of the holomorphic 
discrete series $\pi_{\gamma+\rho_{\frak{k}}}$ of $K_0^*$.  Note that any holomorphic discrete series representation of  
$K_0^*$ is $L_0$-admissible. (It is even $T_0$-admissible; see, for example \cite{repka}).

Combining Theorems \ref{main} and \ref{nonquaternionic},  
we see that there are infinitely many $L_0$-types common 
to $\pi_{\gamma+\rho_{\frak{g}}}$ and $\pi_{\gamma+\rho_{\frak{k}}}$ whenever $S^m(\frak{u}_{-1})$ has a one-dimensional $L_0$-submodule for some $m\geq 1$ and 
$w^0_{\frak{k}}(\Delta_0)=\Delta_0$. We are led 
to the 
following questions.

\noindent
{\bf Questions:} Suppose that there exist infinitely many common $L_0$-types between a Borel-de Siebenthal discrete series $\pi_{\gamma+\rho_{\frak{g}}}$ of $G_0$ and the holomorphic $\pi_{\gamma+\rho_{\frak{k}}}$ of $K_0^*$. Then (i) Does there exist  
a one-dimensional $L_0$-subrepresentation in $S^m(\frak{u}_{-1})$? (ii) Is it true that  $w^0_{\frak{k}}(\Delta_0)=\Delta_0$?

We make use of the description of the $K_0$-finite part of  the Borel-de Siebenthal 
discrete series obtained by \O rsted and Wolf, in terms of the Dolbeault cohomology 
of the flag variety $K_0/L_0$ with coefficients in the holomorphic bundle associated to the $L_0$-represenation $E_{\gamma}\otimes S^m(\frak{u}_{-1})$. This will be recalled in \S2.
Proof of Theorem \ref{nonquaternionic} 
crucially makes use of Theorem~\ref{schmid} on the decomposition of the $L_0$-representation $S^m(\frak{u}_{-2})$ and Littelmann's path model \cite{littelmann},\cite{littelmann95}.   

There are three major obstacles in obtaining complete result in the non-quaternionic case, namely, 
(i) the decomposition of $S^m(\frak{u}_{-1})$ into $L_0$-types $E_\lambda$, (ii)  the decomposition of the tensor product $E_\gamma\otimes E_\lambda$ into irreducible $L_0$-representations $E_\kappa$, and, 
(iii) the decomposition of  
the restriction of the irreducible $K_0$-representation $H^s(K_0/L_0;\mathbb{E}_\kappa)$ to $L_0$. The latter two  
problems can, in principle, be solved using 
the work of Littelmann \cite{littelmann}.  The problem 
of detecting occurrence of an infinite family of common $L_0$-types in the general case appears to be intractable. 

We assume familiarity with  
basic facts concerning symmetric spaces and the theory of discrete series representations, referring the reader to \cite{helgason} and  \cite{knapp}.   

The results of this paper have been announced in~\cite{prs:cr}.

 \noindent
{\bf Acknowledgments:}  The authors thank Prof. R. Parthasarathy for bringing to our notice his paper 
\cite{parthasarathy}  and the referee for a careful reading of the manuscript.
\begin{center}
\vfill\eject
{\bf List of Notations}
\end{center}
\nopagebreak
\begin{tabular}{ll}
$G_0$ & simply connected non-compact real simple Lie group.\\
$K_0$ & maximal compact subgroup of $G_0$.\\
$T_0$ & maximal torus of $K_0.$\\
$\frak{g}_0,\frak{k}_0,\frak{t}_0$ & Lie algebras of $G_0, K_0, T_0$ respectively.\\
$\frak{g}, \frak{k}, \frak{t}$ & complexifications of $\frak{g}_0, \frak{k}_0, \frak{t}_0$ respectively.\\
$G,K$& simply connected complex Lie groups with Lie algebras $\frak{g}$ and $\frak{k}$ respectively.\\
$\Delta$ & root system of $\frak{g}$ with respect to $\frak{t}$.\\
$\Delta^+, \Psi$ & Borel-de Siebenthal positive system of $G_0$ and the set of simple roots.\\
$\nu, \mu$ & the simple non-compact root and the highest root in $\Delta^+$ respectively.\\
$\frak{k}_1^\bc, \frak{k}_1, K_1$ & the simple ideal in $\frak{k}$ containing the root space $\frak{g}_\mu$, compact real form of 
$\frak{k}_1^\bc$ \\ & contained in $\frak{k}_0$ and the simple factor of $K_0$ with Lie algebra $\frak{k}_1$ respectively.\\
$\Delta_i\subset \Delta$ & roots with coefficient of $\nu$ equal to $i$ when expressed in terms of simple roots.\\
$\Delta^+_0, \Delta_0^-$ & positive and negative roots in $\Delta_0.$\\
$\Delta_{\frak{k}}$ & $\Delta_0\cup \Delta_2\cup \Delta_{-2}$, the root system of $\frak{k}$.\\
$\Delta_{\frak{k}}^+, \Psi_{\frak{k}}$ & $\Delta_0^+\cup \Delta_2$ the induced positive system of $\frak{k}$ and the set of simple roots of $\frak{k}$.\\
$\epsilon$ & the simple root  in $\Delta_\frak{k}^+$ which is in $\Delta_2$.\\
$\nu^*$ & fundamental weight of $\frak{g}$ corresponding to $\nu\in \Psi$.\\ 
$\epsilon^*$ & fundamental weight of $\frak{k}$ corresponding to $\epsilon\in \Psi_{\frak{k}}$.\\
$\frak{l}_0, \frak{l}$ & the Lie subalgebra of $\frak{k}_0$ containing $\frak{t}_0$ with root system $\Delta_0$, and its\\ &  complexification.\\
$L_0, L_0'$, $L$ & the Lie subgroups of $K_0$ and $K$ with Lie algebras $\frak{l}_0, [\frak{l}_0,\frak{l}_0],$ and $\frak{l}$ respectively.\\ 
$K_0^*$ & the real form of $K$ dual to the compact form $K_0$ with respect to $L_0$. \\
$w^0_{\frak{k}}, w^0_{\frak{l}}$ & longest element 
of the Weyl group of $K_0$ and $L_0$.\\
$\frak{u}_i$ & $\sum_{\alpha\in \Delta_i}\frak{g}_\alpha$, $i=\pm 1,\pm 2$.\\
$Q, \frak{q}$ & the parabolic subgroup of $G$ with Lie algebra $\frak{q}=\frak{l}+\frak{u}_{-1}+\frak{u}_{-2}.$\\
$\mathcal{A}(E, L)$ & the algebra of relative invariants of a prehomogeneous $L$-represenation $E$.\\
$Y, s$ & the flag variety $K_0/L_0=K/K\cap Q$,  $s=\dim_\bc Y.$\\
$X$ & $K_0^*/L_0$, the non-compact dual of $Y$.\\
$w_Y^0$ & the element  $w^0_{\frak{k}}w_{\frak{l}}^0$.\\
$\rho_{\frak{g}}, \rho_{\frak{k}}$ & $(1/2)(\sum_{\alpha\in \Delta^+}\alpha), 
(1/2)(\sum_{\alpha\in \Delta_{\frak{k}}^+}\alpha).$\\
$\pi_{\gamma+\rho_{\frak{g}}}, \pi_{\gamma+\rho_{\frak{k}}}$ & discrete series representations of $G_0$ and $K_0^*$ with 
Harish-Chandra\\ & parameters  $\gamma+\rho_{\frak{g}}, \gamma+\rho_{\frak{k}}$ respectively.\\
$\pi_{K_0}$ & the space of $K_0$-finite vectors of a $G_0$-representation $\pi$. \\
$E_\kappa, V_\lambda$ & the irreducible $\frak{l}$  or $L_0$ (resp. $\frak{k}$ or $K_0$) representation with highest weight\\ &$\kappa$ (resp. $\lambda$).\\
$\res_{\frak{l}} V_\lambda$ & restriction of $V_\lambda$ to $\frak{l}$.\\
$U_k$ & irreducible $\frak{su}(2)$-representation of dimension $k+1$.\\
$\mathbb{E_\kappa}$ & the holomorphic vector bundle over $Y$ associated to $E_\kappa$.\\ 
$\{\gamma_1,\ldots,\gamma_r\}$ & maximal set of strongly orthogonal non-compact negative roots of $K_0^*$.\\
\end{tabular}

%%%%%%%%%%%%%%%%%%%%%%%%%%%%%%%%%%%%%%%%%%%%%%%%%

%\input{material}
%%%%%%%%%%%%%%%%%%%%%%%%%%%%%%%%%%%%%%%%%%%%%%%%%%

\section{Borel-de Siebenthal discrete series}

In this section we recall a description of the 
Borel-de Siebenthal series.  
We shall follow the notations of \O rsted and Wolf, which we now recall.

\subsection{} \label{hypothesis}
Let $\frak{g}_0$ be a real simple non-compact 
Lie algebra and let $\frak{k}_0$ be a maximal compactly embedded Lie subalgebra of $\frak{g}_0$ with rank $\frak{g}_0=$ rank $\frak{k}_0$ and $\frak{k}_0$ semisimple. 

Let $\frak{t}_0$ 
be a Cartan subalgebra of $\frak{k}_0$, which is also 
a Cartan subalgebra of $\frak{g}_0.$   The notations 
$G_0,K_0, \frak{g}, \frak{k},$ etc. will have the same 
meaning as in \S1. Let $\Delta$ be the root system of $(\frak{g},\frak{t})$, 
$\Delta^+\subset \Delta$ be a Borel-de Siebenthal positive system and $\Psi$ the set of simple roots.  Let $\alpha\in \Delta$ be any root and let $n_\nu(\alpha)$ be the coefficient of $\nu$ (the non-compact simple root) when $\alpha$ is expressed as a sum of simple roots. Since $\frak{k}_0$ is semisimple, one has a partition of the set of roots $\Delta$ into subsets $\Delta_i, i=0, \pm 1,\pm 2$ where $\Delta_i\subset \Delta$ defined to be $\{\alpha\in \Delta \mid n_\nu(\alpha)=i\}$. Denote by $\mu$ the highest root; then $\mu\in\Delta_2$.   The set $\Delta_{\frak{k}}:=\Delta_0\cup \Delta_2\cup \Delta_{-2}$ is the root system of $\frak{k}$ with respect to $\frak{t}$ for which 
$\Psi\setminus \{\nu\}\cup \{-\mu\}$ is a set of 
simple roots defining a positive system of roots, namely, $\Delta_0^+\cup \Delta_{-2}$. On the other hand $(\frak{k}, \frak{t})$ inherits a positive root system from $(\frak{g},\frak{t})$, namely, 
$\Delta_{\frak{k}}^+:=\Delta_0^+\cup\Delta_2$.    
Lemma \ref{positivesystem} brings out the relation between the two.  

The Killing form $B:\frak{t}\times \frak{t}\lr \bc$ determines a non-degenerate symmetric bilinear pairing  $\langle ~,~\rangle:\frak{t}^*\times \frak{t}^*\lr \bc$ which is normalized so that 
$\langle\nu,\nu\rangle=2.$   For any $\alpha\in \frak{t}^*$, denote by $H_\alpha\in \frak{t}$ the unique element such that $\alpha(H)=B(H,H_\alpha)$.  Then our  normalization requirement is that 
$\langle\alpha,\beta\rangle:=2B(H_\alpha,H_\beta)/B(H_\nu,H_\nu)$ for all $\alpha,\beta\in \frak{t}^*$.  Let $\nu^*\in \frak{t}^*$ be the 
fundamental weight corresponding to $\nu\in \Psi$. 

Now define  $\frak{q}:=\frak{t}+\frak{u}_0+\frak{u}_{-1}+\frak{u}_{-2}$ where $\frak{u}_i=\sum_{\alpha\in 
\Delta_i}\frak{g}_\alpha, -2\leq i\leq 2$. Then $\frak{q}$ is a maximal parabolic subalgebra of $\frak{g}$ that omits the 
non-compact simple root $\nu$.  The Levi part of $\frak{q}$ is the Lie subalgebra $\frak{l}=\frak{t}+\frak{u}_0$ and the nilradical of $\frak{q}$ is $\frak{u}_-=\frak{u}_{-1}
+\frak{u}_{-2}$.  Note that the centre of $\frak{l}$ is $\mathbb{C}H_{\nu^*}$.  We have that 
$\Delta_{\frak{l}}:=\Delta_0$ is the root system of $\frak{l}$ with respect to $\frak{t}\subset \frak{l}$ for which $\Psi\setminus\{\nu\}$ is the set of simple roots 
defining the positive system $\Delta_{\frak{l}}^+:=\Delta^+_0$. 
Let $\frak{k}_1^\bc$ denote the simple ideal of $\frak{k}$   
that contains the root space $\frak{g}_\mu$.  It is the complexification of the Lie algebra  $\frak{k}_1$ of 
a compact Lie group $K_1$ which is a simple factor of $K_0$. 
It turns out that $\frak{u}_{2},\frak{u}_{-2}\subset \frak{k}_1^\bc$.
Let $\frak{k}_2$ be the ideal of $\frak{k}_0$ such that  $\frak{k}_0=\frak{k}_1\oplus \frak{k}_2$. We let $\frak{l}_j^\bc
=\frak{k}_j^\bc\cap \frak{l}, j=1,2.$ Note that $\frak{k}_2^\bc=\frak{l}_2^\bc$ and so $\frak{l}_2^\bc$ is semisimple. Thus the centre of $\frak{l}$ is contained in $\frak{l}_1^\bc$.   

Let $G$ denote the simply connected complex Lie group 
with Lie algebra $\frak{g}$, $Q\subset G$, the parabolic subgroup with Lie algebra $\frak{q}$. Denote by $K, L\subset G$ the connected Lie subgroups with Lie algebras $\frak{k}, \frak{l}$ respectively.  Let $L_0\subset K_0$ be the 
 centralizer of the circle group $S_{\nu^*}:=\{\exp(itH_{\nu^*})\mid t\in \br\}$ contained in $K_0$.   
Then $K_0/L_0$ is a complex flag variety which is 
a Hermitian symmetric space. 
Also $\frak{l}_0\subset \frak{k}_0$ is a compact real form 
of $\frak{l}$.  Let $L_1\subset K_1$ be the centralizer of 
$S_{\nu^*}\subset K_1$.  Then $L_1\subset L_0$ and $Lie(L_1)=:\frak{l}_1$ is a compact real form of $\frak{l}_1^\bc$.  Let $K_2$ be the connected Lie subgroup of 
$K_0$ with Lie algebra $\frak{k}_2$. Then $K_0=K_1\times K_2$ as $K_0$ is simply connected. Also $L_0=L_1\times K_2.$ It 
will be convenient to set $L_2:=K_2$.

The inclusion  $\frak{g}_0\hookrightarrow \frak{g}$ induces a map 
$G_0\lr G$, which defines smooth maps $G_0/L_0\subset G/Q$  
and $K_0/L_0\subset G_0/L_0\subset G/Q$ since $\frak{l}_0\subset \frak{q}$. Since 
$\dim_\br(G_0/L_0)=\dim_\br(\frak{u}_1+\frak{u}_2)=
2\dim_\bc(G/Q)$, we conclude that $G_0/L_0$ is an open 
domain of the complex flag variety $G/Q$.  
Note that one has a fibre bundle projection $G_0/L_0\lr G_0/K_0$ with fibre $K_0/L_0$. We shall denote the identity coset 
of any homogeneous space by $o$.  
The holomorphic tangent bundles of 
$K_0/L_0$ and $G/Q$ are the bundles associated to the $L_0$-modules $\frak{u}_2$ and $\frak{u}_1\oplus \frak{u}_2$ respectively since we have the  
isomorphisms of tangent spaces $\mathcal{T}_o K_0/L_0=\frak{u}_2$ and $\mathcal{T}_o G/Q=\frak{u}_1\oplus \frak{u}_2$ of $L_0$-modules.  Hence the normal bundle to the imbedding $K_0/L_0\hookrightarrow G/Q$ is the bundle associated to the representation of $L_0$ on $\frak{u}_1$.

Denote by $(K_0^*, L_0)$ the non-compact Hermitian symmetric pair dual to the compact Hermitian symmetric pair $(K_0,L_0)$.  A well-known result of Harish-Chandra 
\cite[Ch. VIII]{helgason} is that $K_0^*/ L_0$ is naturally  imbedded as a bounded symmetric domain in $\frak{u}_2=\mathcal{T}_o(K_0/L_0),$ the holomorphic tangent space at $o$ of $K_0/L_0$.  Denote by $\mathcal{U}_{\pm 2}\subset K$ the image of $\frak{u}_{\pm 2}$ under the exponential map. 
Then $\mathcal{U}_2$ is an open neighbourhood of $o$ in $K/(L.\mathcal{U}_{-2})\cong K_0/L_0$. Thus $K_0^*/L_0=:X$ is imbedded in $K_0/L_0=:Y$ as an open complex analytic submanifold.   

We recall the following result due to \O rsted and Wolf \cite{ow}. See also \cite{parthasarathy} and Appendix~2~(\S\ref{s:appendixrp}) below.  Let $\gamma$ be the highest weight of an 
irreducible finite dimensional complex representation 
of $L_0$ on $E_\gamma$ and suppose that $\langle \gamma+\rho_{{\frak{g}}},\alpha\rangle <0$ for all $\alpha\in 
\Delta_1\cup \Delta_2$. 
 
\begin{theorem} \label{bds} (Parthasarathy \cite{parthasarathy}, \O rsted and Wolf  \cite{ow})  The $K_0$-finite part of the Borel-de Siebenthal discrete 
series $\pi_{\gamma+\rho_{\frak{g}}}$ is isomorphic to  
$\oplus_{m\geq 0} H^s(Y;\mathbb{E}_\gamma\otimes 
\mathbb{S}^m(\frak{u}_{-1}))$ where $s=\dim Y$ and moreover, it is $K_1$-admissible. 
\end{theorem}   

The $K_1$-admissibility of the Borel de Siebenthal discrete series also follows from Kobayashi~\cite{kobannals} who obtained a criterion for the admissibility of the restriction of certain representations to reductive subgroups in a more general context.

\subsection{Certain $L_0$ representations}\label{ss:lnot}
Since $\frak{l}=\frak{l}_1^\bc\oplus \frak{l}_2^\bc$, we have the decomposition 
\label{page}
$\gamma=\gamma_1+\gamma_2,$ with  $\gamma_i\in \frak{t}_i^*$ where  $\frak{t}_i=\frak{l}_i^\bc\cap \frak{t}$.  Also, 
$E_\gamma=E_{\gamma_1}\otimes E_{\gamma_2}$. Furthermore $H_{\nu^*}$ generates the centre 
of $\frak{l}_1^{\mathbb{C}}$ and we have the Levi decomposition  
$\frak{l}_1^\bc=[\frak{l}_1^\bc,\frak{l}_1^\bc]+\frak{z}(\frak{l}^\bc_1)$ where $\frak{z}(\frak{l}_1^\bc)=\frak{z}(\frak{l})=\bc H_{ \nu^*}$. 
We write $\gamma_1=\gamma'+t\nu^*$ where $\gamma'\perp \nu^*$.  
The assumption that $\gamma$ is an $\frak{l}$-dominant integral weight and that $\gamma+\rho_{\frak{g}}$ is negative on positive roots of $\frak{g}$ complementary to those of $\frak{l}$ implies that $t$ is `sufficiently negative'.  That is, $t$ is real and 
it satisfies the conditions (see \cite[Theorem 2.12]{ow}):
\[
t<-1/2\langle\gamma_0+\rho_{\frak{g}},\mu\rangle ~~ \textrm{and} ~~~t<-\langle\gamma_{0}+\rho_{\frak{g}},w^0_{\frak{l}}(\nu)\rangle \eqno(1)
\]  
where $\gamma_0:=\gamma-t\nu^*\in [\frak{l},\frak{l}]$ and $w_{\frak{l}}^0$ denotes the longest element of the 
Weyl group of $(\frak{l},\frak{t})$ with respect to $\Delta^+_{\frak{l}}$.  \footnote{The decomposition 
of $\gamma=\gamma_0+t\nu^*$ used in \cite[Theorem 2.12]{ow} 
is different.}

Recall that $\Delta_{\frak{k}}^+=\Delta_0^+\cup \Delta_2$ is the positive root system of $(\frak{k},\frak{t})$ that is 
compatible with the positive root system $\Delta^+$ of $\frak{g}$. It is easily seen that $\Psi_{\frak{k}}:=\Psi\setminus\{\nu\}\cup \{\epsilon\}\subset \Delta_{\frak{k}}^+$ is the set of simple roots where 
$\epsilon$ is the lowest root in $\Delta_2$ (so that $\beta\geq \epsilon$ for all $\beta\in \Delta_2$).
 \footnote{\O rsted and Wolf \cite{ow} denote by $\Psi_{\frak{k}}$ the set $\Psi\setminus \{\nu\}\cup \{-\mu\}$.} Also $\Psi_{\frak{l}}:=\Psi\cap \Delta_0^+=\Psi\setminus\{\nu\}$ is the set of simple 
roots of $\frak{l}$ for the positive system $\Delta_0^+$.      
It is readily verified that 
$\Psi_{\frak{k}}=w_Y(\Psi\setminus\{\nu\}\cup \{-\mu\})$   
where $w_Y=w_{\frak{k}}^0w^0_{\frak{l}}.$ 
The adjoint action of $L_0$ on $\frak{g}$ yields $L_0$-representations on $\frak{u}_i, i=\pm 1, \pm 2$, which are irreducible. The highest (resp. lowest) weights of $\frak{u}_{-2}, \frak{u}_{-1}$, $j=1,2$, are $-\epsilon, -\nu$ (resp. $-\mu, w^0_{\frak{l}}(-\nu)$) respectively.   

Let $\Xi=\{\xi_1,\ldots,\xi_l\}$ be the 
set of fundamental weights of $\frak{g}$ with respect to $\Psi=\{\psi_1,\ldots,\psi_l\}$ so that $2\langle\xi_i,\psi_j\rangle/\langle\psi_j,\psi_j\rangle=\delta_{i,j}.$ (Here $\delta_{i,j}$ denotes the Kronecker 
delta.)  

If $\psi\in \Psi_{\frak{k}}$,  the corresponding fundamental weight of $\frak{k}$ will be denoted by $\psi^*$. If $\psi_i$ is a compact simple root of $\frak{g}_0$, it should be noted that in general 
$\psi_i^*\neq \xi_i$.      

In conformity with the notations of \cite{ow}, we shall denote by $\nu^*$ the weight $\xi_{i_0}$ where $\nu=\psi_{i_0}\in \Psi$. (Since $\nu\notin \Psi_{\frak{k}}$ there is no danger of confusion.)  
 
\begin{lemma} \label{positivesystem}
With the above notations, suppose that $\nu=\psi_{i_0}$  
and $\epsilon=\sum a_i\psi_i\in \Delta_2$. Then:
(i)  $\epsilon^*=||\epsilon||^2\nu^*/4$ and $\psi_i^*=\xi_i-a_i||\psi_i||^2\nu^*/4, i\neq i_0$.\\
(ii) $w_Y(\Delta_0^+\cup \Delta_{-2})=\Delta_0^+\cup\Delta_{2}$,
$\Psi_{\frak{k}}=w_Y(\Psi\setminus\{\nu\}\cup\{-\mu\})$.\\
(iii) If $\lambda\in \frak{t}^*$, then  $\lambda=\lambda'+a\nu^*$ where  
$a=\langle \lambda,\nu^*\rangle/||\nu^*||^2$ and $\lambda' \in (\frak{t} \cap [\frak{l},\frak{l}])^*=\{\nu^*\}^\perp$. \\ 
(iv)  The sum $\sum_{\beta\in \Delta_2}\beta=c\epsilon^*$ where $c=s||\epsilon||^2/2||\epsilon^*||^2$ (with $s=|\Delta_2|$) is an integer.  
\end{lemma}
\begin{proof}
We will only prove (iv), the proofs of the remaining parts being straightforward.  

Observe that if $E$ is a finite dimensional representation 
of $\frak{l}$, then the sum $\lambda$ of all weights of $E$, counted with multiplicity, is a multiple of $\epsilon^*$. 
This follows from the fact that the top-exterior $\Lambda^{\dim (E)}(E)$ is a one dimensional representation of $\frak{l}$ isomorphic to $\bc_\lambda$.  Applying this 
to $\frak{u}_{2}$, we obtain that $\sum_{\beta\in \Delta_2}\beta=c\epsilon^*$.  Clearly $c$ is an integer 
since the $\beta$ are roots of $\frak{k}$ and so $\sum_{\beta\in\Delta_2}\beta$ is in the weight lattice. 
\end{proof}

\noindent 
{\it Example}: Consider the group $G_0=Sp(2,1)$.  The 
non-compact root in the Bourbaki root order of 
$\frak{sp}(3,\bc)$ is $\nu=\psi_2$.  Also $K_0=Sp(2)\times Sp(1), K_1=Sp(2), L_1=U(2), L_2=K_2=Sp(1), \Delta_0^+=\{\psi_1,\psi_3\}, \Delta_1=\{\psi_2,\psi_1+\psi_2,\psi_1+\psi_2+\psi_3,\psi_2+\psi_3\}, 
\Delta_2=\{\mu=2\psi_1+2\psi_2+\psi_3, \psi_1+2\psi_2+\psi_3, 2\psi_2+\psi_3=\epsilon\}.$  Furthermore, 
$\Psi_{\frak{k}}=\{\psi_1,\psi_3,\epsilon\}$ where $\langle\psi_3,\epsilon\rangle=0, \langle\psi_1,\epsilon\rangle = -2,$ and, $\psi_1^*=\xi_1, \psi_3^*=\xi_3-\xi_2$ and $\epsilon^*=\nu^*$.  Finally $c=3.$ 
%%%%  To be verified %%%%

\begin{remark}\label{spin}{\em 
(i) The parity of $c$ will be relevant for our purposes. We 
give an interpretation of it in terms of the existence of spin structures on $Y=K_0/L_0$.  
Recall that $\pi_1(T_0)$ is the kernel of the exponential map $\frak{t}_0\to T_0$.  Thus 
$H^1(T_0;\mathbb{Z})\cong \hom(\pi_1(T_0),\mathbb{Z})$ is a lattice in $\frak{t}_0^*$.   This is the 
weight lattice of $\frak{k}_0$ (with respect to $\frak{t}_0$) since 
$K_0$ is simply connected.   
The cohomology group $H^2(Y;\mathbb{Z})$ is naturally isomorphic to $\mathbb{Z}[\epsilon^*]\cong \mathbb{Z},$
the quotient of the weight lattice of $\frak{k}_0$ by the weight lattice of $[\frak{l}_0,\frak{l}_0]$.    (See \cite[\S14.2]{bh}.)
If $\lambda$ is a weight of $\frak{k}_0$ its class in $H^2(Y;\bz)$ 
is denoted by $[\lambda]$.  Thus $[\lambda]=2(\langle \lambda, \epsilon\rangle/||\epsilon||^2) [\epsilon^*]$.  
The holomorphic tangent bundle $\mathcal{T}Y$ is the bundle associated to the $L_0$-representation $\frak{u}_{2}=\sum_{\beta\in \Delta_2} \frak{g}_\beta$.  This implies that $c_1(Y)$, first Chern class of $Y$,  
equals $\sum_{\beta\in \Delta_2}[\beta]=c[\epsilon^*]\in H^2(Y;\mathbb{Z})$.   
Consequently $Y$ admits a spin structure if and only if $c$ is even.  The value of $c$ can be explicitly 
computed.  
(See, for example, \cite[\S 16]{bh}.) This leads to the following conclusion. 
The complex Grassmann 
variety $\bc G_p(\bc^{p+q})=SU(p+q)/S(U(p)\times U(q))$ admits a spin structure if and only if $p+q$ is even and that the complex quadric $SO(2+p)/SO(2)\times SO(p)$ admits a spin structure 
precisely when $p$ is even. The orthogonal Grassmann 
variety $SO(2p)/U(p)$ admits a spin structure for all $p.$ 
The symplectic Grassmann variety $Sp(p)/U(p)$ admits a spin structure if and only if $p$ is odd.  
The Hermitian 
symmetric spaces $E_6/(Spin(10)\times SO(2))$ and 
$E_7/(E_6\times SO(2))$ admit spin structures.\\

(ii) The highest weight of any irreducible 
$L_0$-submodule of $E_\gamma\otimes S^m(\frak{u}_{-1})$ is of the form $\gamma+\phi$ where $\phi$ is a weight of $S^m(\frak{u}_{-1})$. Thus $\phi=\alpha_1+\cdots+\alpha_m$ for suitable $\alpha_i$ in $\Delta_{-1}$ (not necessarily distinct). 
Now if $\alpha \in \Delta_{-1}$ and $\beta \in \Delta_2$, then $\beta - \alpha$ is not a root. Hence $\langle \alpha, \beta\rangle \leq 0$ for all $\alpha \in \Delta_{-1}, \beta \in \Delta_2$. It follows that $\langle\gamma+\rho_{\frak{k}},\beta\rangle \leq \langle \gamma+\rho_{\frak{g}},\beta\rangle $ and $\langle \phi,\beta\rangle\leq 0$ for all $\beta\in \Delta_2$. Since $\langle \gamma +\rho_\frak{g},\beta\rangle < 0$ for all $\beta\in \Delta_2$, therefore $\langle \gamma+\rho_{\frak{k}}, \beta \rangle <0$ and $\langle \gamma+\phi+\rho_{\frak{k}},\beta\rangle <0$ for all $\beta \in \Delta_2$. Hence, by the Borel-Weil-Bott theorem, 
the highest weight of $H^s(Y; \mathbb{E}_{\gamma+\phi})$ equals $w_Y(\gamma+\phi+\rho_{\frak{k}})-\rho_{\frak{k}}$. 
We shall make use of this remark in the sequel without explicit reference 
to it.  }
\end{remark}

\subsection{Classification of Borel-de Siebenthal root orders}\label{ss:bdsrorder}
The complete classification of Borel-de Siebenthal root orders is given in \cite[\S3]{ow}.  
For the convenience of the reader we recall here, 
in brief, their classification.  
%We list the quaternionic and non-quaternionic cases separaterly in \S\ref{s:appendix1} (Appendix~1).

Let $\frak{g}_0$ be a non-compact real simple Lie algebra satisfying the conditions of \ref{hypothesis}. Having fixed a fundamental Cartan subalgebra  $\frak{t}_0\subset \frak{g}_0$; a positive 
root system of $(\frak{g},\frak{t})$ containing exactly one non-compact simple root $\nu$, is Borel-de Siebenthal if the coefficient of $\nu$ in the highest root is 2. Conversely, let $\frak{g}$ be a complex simple Lie algebra. Choose a Cartan subalgebra $\frak{t}\subset \frak{g}$ and a positive root system of $(\frak{g}, \frak{t})$. If there exists a simple root $\nu$ whose coefficient in the highest root is 2, then $\nu$ determines uniquely (up to an inner automorphism) a non-compact real form $\frak{g}_0$ of $\frak{g}$ satisfying the conditions of \ref{hypothesis} such that the positive system is a Borel-de Siebenthal positive system of $\frak{g}_0$.

If $\Psi$ is  the set of simple roots of a Borel-de Siebenthal positive system of $\frak{g}_0$ and $\nu \in \Psi$ is the unique non-compact root, we denote the Borel-de Siebenthal root order by $(\Psi, \nu)$. Corresponding to $\frak{g}_0$, we can have several Borel-de Siebenthal root orders. Given one such, we have its negative $(-\Psi, -\nu)$.   The Borel-de Siebenthal root orders up to sign changes
are tabulated in Appendix~1 (\S\ref{s:appendix1}).

The quaternionic case is characterized by the property 
that highest root $\mu$ is orthogonal to {\it all} the compact simple roots and hence $-\mu$ is adjacent to the simple non-compact root $\nu$ in the extended Dynkin diagram of $\frak{g}$.  

\ignore{    %%%%%%%%%%%%%%%%%%%%%%%%
(a)  {\it Quaternionic type:} We have $\frak{k}_1=\frak{su}(2), \frak{l}_1=\frak{so}(2)=i\mathbb{R}\nu^*$.  Also $Y=\mathbb{P}^1$. $X=Y^*=SU(1,1)/U(1),$ the unit disk in $\bc$.  The condition $w^0_{\frak{k}}(\epsilon)=-\epsilon$ is trivially valid. 
%In this case $Y$ admits a spin structure. 
\begin{enumerate}
\item $\frak{g}_0=\frak{so}(4,2l-3), l>2$.  Then $\frak{g}$ is of type $B_l$ and $\nu=\psi_2$. $\frak{l}_2=\frak{sp}(1)\oplus \frak{so}(2l-3). \ \  \mathcal{A}=\bc[f], 
\deg(f)=4.$   \\

\item $\frak{g}_0=\frak{so}(4,1)$.  Then $\frak{g}$ is of 
type $B_2$, $\nu=\psi_2, \frak{l}_2=\frak{sp}(1)$. 
$\mathcal{A}=\bc.$ \\

\item  $\frak{g}_0=\frak{sp}(1,l-1), l>1.$ Then $\frak{g}$ is 
of type $C_l, \nu=\psi_1, \frak{l}_2=\frak{sp}(l-1)$. 
$\mathcal{A}=\bc. $\\

\item $\frak{g}_0=\frak{so}(4, 2l-4), l>4$.  $\frak{g}$
is of type $D_l, \nu=\psi_2, \frak{l}_2=\frak{sp}(1)\oplus 
\frak{so}(2l-4). \ \ \mathcal{A}=\bc[f], \deg(f)=4$. \\

\item $\frak{g}_0=\frak{so}(4,4)$. $\frak{g}$ is of type 
$D_4$, $\nu=\psi_2, \frak{l}_2=\frak{sp}(1)\oplus \frak{sp}(1)\oplus \frak{sp}(1).\ \  \mathcal{A}=\bc[f], \deg(f)=4.$\\

\item $\frak{g}_0=\frak{g}_{2; A_1,A_1}$, the split real form of the exceptional Lie algebra of type $G_2$. 
$\frak{g}=\frak{g}_2, \nu=\psi_2, \frak{l}_2=\frak{sp}(1).\ \  \mathcal{A}=\bc[f], \deg(f)=4.$\\

\item $\frak{g}_0=\frak{f}_{4;A_1,C_3},$  the split 
real form of the exceptional Lie algebra of type $F_4$.   
$\frak{g}=\frak{f}_4, \nu=\psi_1, \frak{l}_2=\frak{sp}(3). \ \ \mathcal{A}=\bc[f], \deg(f)=4.$\\

\item $\frak{g}_0=\frak{e}_{6;A_1,A_5,2}$.  $\frak{g}=\frak{e}_6$, the exceptional Lie algebra. $\nu=\psi_2, 
\frak{l}_2=\frak{su}(6). \ \ \mathcal{A}=\bc[f], \deg(f)=4.$ \\

\item $\frak{g}_0=\frak{e}_{7;A_1, D_6,1}$. $\frak{g}=
\frak{e}_7, \nu=\psi_1, \frak{l}_2=\frak{so}(12). \ \ 
\mathcal{A}=\bc[f], \deg(f)=4.$\\

\item $\frak{g}_0=\frak{e}_{8; A_1, E_7}$. $\frak{g}=\frak{e}_8, \nu=\psi_8, \frak{l}_2^{\mathbb{C}}=\frak{e}_7. \ \ \mathcal{A}=\bc[f], \deg(f)=4.$ 
\end{enumerate}

\noindent
(b) {\it  Non-quaternionic type:}
\begin{enumerate}
\item $\frak{g}_0=\frak{so}(2p, 2l-2p+1)$, $2<p<l, l> 3$. $\frak{g}$ is of type $B_l$, $ \nu=\psi_p, \frak{k}_1=\frak{so}(2p), \frak{l}_1=\frak{u}(p), \frak{l}_2=\frak{s0}(2l-2p+1)$.  The variety $Y=SO(2p)/U(p),  X=SO^*(2p)/U(p)$.    $w^0_{\frak{k}}(\epsilon)=-\epsilon$ if and only if $p$ is even.  $\mathcal{A}=\bc[f]$ (with $\deg(f)=2p$) if and only if $3p\leq 2l+1$. \\

\item $\frak{g}_0=\frak{so}(2l,1), l>2$. $\frak{g}$ is of type $B_l$,  $\nu=\psi_l, \frak{k}_0=\frak{k}_1=\frak{so}(2l), \frak{l}_1=\frak{u}(l).$  The variety $Y=SO(2l)/U(l), X=SO^*(2l)/U(l)$. $w^0_{\frak{k}}(\epsilon)=-\epsilon$ if and only if $l$ is even.  $\mathcal{A}=\bc$. \\

\item $\frak{g}_0=\frak{sp}(p, l-p)$, $l> 2, 1<p<l.$  $\frak{g}$ is of type $C_l$, $ \nu=\psi_p, \frak{k}_1=\frak{sp}(p), \frak{l}_1=\frak{u}(p), \frak{l}_2=\frak{sp}(l-p)$, and $Y=Sp(p)/U(p), X=Sp(p,\mathbb{R})/U(p). \ \ 
w^0_{\frak{k}}(\epsilon)=-\epsilon$. $\mathcal{A}=\bc[f]$, (with $\deg(f)=p$) 
if and only if $3p\leq 2l$ and $p$ even.\\

\item $\frak{g}_0=\frak{so}(2l-4, 4), l> 4$. $\frak{g}$ is of type $D_l$, $ \nu=\psi_{l-2}, \frak{k}_1=\frak{so}(2l-4), \frak{l}_1=\frak{u}(l-2), \frak{l}_2=\frak{su}(2)\oplus \frak{su}(2)$.  The variety $Y=SO(2l-4)/U(l-2), 
X=SO^*(2l-4)/U(l-2).\ \ w^0_{\frak{k}}(\epsilon)=-\epsilon$ if and only if $l$ is even.   $\mathcal{A}=\bc$ if $l>6$. When $l=5,6$, $\mathcal{A}
=\bc[f]$ with $\deg(f)=6, 8$ respectively. 
\\

\item $\frak{g}_0=\frak{so}(2p,2l-2p), 2<p<l-2, l>5.$  $\frak{g}$ is of type $D_l$, $\nu=\psi_p$, $\frak{k}_1=
\frak{so}(2p), \frak{l}_1=\frak{u}(p)$, $\frak{l}_2=\frak{so}(2l-2p)$. $Y=SO(2p)/U(p), X=SO^*(2p)/U(p). \ \ w^0_{\frak{k}}(\epsilon)=-\epsilon$ if and only if $p$ is even. 
$\mathcal{A}=\bc[f]$ (with $\deg(f)=2p$) if and only if $3p\leq 2l$.\\

\item $\frak{g}_0=\frak{f}_{4;B_4},$ the real form of $\frak{f}_4$ having $\frak{k}_0\cong \frak{so}(9)$ as a maximal compactly embedded subalgebra. $\nu=\psi_4$ and  $\frak{k}_0=\frak{k}_1, \frak{l}_1=i\mathbb{R}\nu^*\oplus \frak{so}(7)$. $Y=SO(9)/(SO(7)\times SO(2)), X=SO_0(2,7)/(SO(2)\times SO(7)). \ \ w^0_{\frak{k}}(\epsilon)=-\epsilon$.  $\mathcal{A}=\bc[f],  \deg(f)=2$. \\
%$Y$ does not admit a spin structure.\\

\item $\frak{g}_0=\frak{e}_{6;A_1,A_5,1}$, a real 
form of $\frak{e}_6$ with $\nu=\psi_3$. $\frak{k}_1=\frak{su}(6), \frak{l}_1=\frak{su}(5)\oplus i\mathbb{R}\nu^*, \frak{l}_2= \frak{su}(2)$. $Y=
\mathbb{P}^5, X=SU(1,5)/S(U(1)\times U(5)). \ \ w^0_{\frak{k}}(\epsilon)\neq-\epsilon$.  $\mathcal{A}=\bc$. \\
%$Y$ admits a spin structure.\\

\item $\frak{g}_0=\frak{e}_{7;A_1,D_6,2}$, a real form 
of $\frak{e}_7$ with $\nu=\psi_6$. $\frak{k}_1=\frak{so}(12), \frak{l}_1=\frak{so}(10)\oplus i\mathbb{R}\nu^*$, $\frak{l}_2=\frak{sp}(1)$. $Y=SO(12)/SO(2)\times SO(10), X=SO_0(2, 10)/(SO(2)\times SO(10)). \ \ w^0_{\frak{k}}(\epsilon)=-\epsilon$. $\mathcal{A}=\bc$.\\
%$Y$ admits a spin structure.  \\

\item $\frak{g}_0=\frak{e}_{7;A_7}$, a real form of 
$\frak{e}_7$ with $\nu=\psi_2$. $\frak{k}_0=\frak{k}_1=\frak{su}(8)$, $\frak{l}_1= \frak{su}(7)\oplus i\mathbb{R}\nu^*$. The variety $Y=\mathbb{P}^7, 
X=SU(1,7)/S(U(1)\times U(7)). \ \ w^0_{\frak{k}}(\epsilon)\neq-\epsilon$.   $\mathcal{A}=\bc[f], \deg(f)=7$.\\
%$Y$ admits a spin structure.\\

\item 
$\frak{g}_0=\frak{e}_{8;D_8}$, a real form of $\frak{e}_8$ with $\nu=\psi_1$. $\frak{k}_0=\frak{k}_1=\frak{so}(16)$, $\frak{l}_1=i\mathbb{R}\nu^*\oplus \frak{so}(14)$. 
$Y=SO(16)/SO(2)\times SO(14), X=SO_0(2,14)/(SO(2)\times SO(14)). \ \ w^0_{\frak{k}}(\epsilon)=-\epsilon$.  
$\mathcal{A}=\bc[f], \deg(f)=8$.
%$Y$ admits a spin structure.
\end{enumerate}
 } %%%%%%%%%%%%%%%%%%%%%%%%%%%%%    End ignore %%%%%%%%%%%%%%%%%%%%%%%%%%%%%%%%

\subsection{Relative invariants of $(\frak{u}_1,L)$}  The action of $L=L_0^\bc$ on $\frak{u}_1$ is known 
to have a Zariski dense orbit.  It follows that the coordinate ring $\bc[\frak{u}_1]=S^*(\frak{u}_{-1})$ 
has no non-constant invariant functions, that is, $\bc[\frak{u}_{1}]^L=\bc$.  However, it is possible that $\frak{u}_1$ has non-zero relative invariants, that is, an $h\in \bc[\frak{u}_1]$ such that $x.h=\chi(x)h, x\in L,$ for some rational character $\chi:L\lr \bc^*$.   It can be seen that the subalgebra $\mathcal{A}(\frak{u}_{1},L)\subset \mathbb{C}[\frak{u}_{1}]$ of all relative invariants is either $\bc$ or is a polynomial algebra $\bc[f]$ for a suitable (non-zero) homogeneous polynomial function $f\in \bc[\frak{u}_1]$.  It is clear that a homogeneous function $h$ belongs to $\mathcal{A}(\frak{u}_1,L)$ if and only if 
$\bc h$ is an $L$-submodule of $S^m(\frak{u}_{-1})$ where $m=\deg(h)$.  \O rsted and Wolf \cite{ow} determined when $\mathcal{A}(\frak{u}_1, L)$ is a polynomial algebra $\bc[f]$   and described in such cases the generator $f$ in detail. See also \cite{sk}.

\begin{proposition} \label{relativeinvariants} 
Let $\Delta^+$ be a Borel-de Siebenthal positive system of $(\frak{g},\frak{t})$ listed above.  
If  $\frak{g}_0=\frak{so}(4,1), \frak{sp}(1,l-1) ( with\  l>1 ), \frak{e}_{6;A_1,A_5,1}, \frak{e}_{7;A_1,D_6,2}, \frak{g}_0=\frak{so}(2p, r)$ with $p>r\geq 1$,  $\frak{g}_0=\frak{sp}(p,q)$ where $p>2q>0$ or $p$ is odd, then $\mathcal{A}(\frak{u}_1, L)=\mathbb{C}$.  In all the remaining cases $\mathcal{A}(\frak{u}_1,L)=\bc[f]$, a  polynomial algebra where $\deg(f)>0$.   \hfill $\Box$

In the case when $\frak{g}_0=\frak{so}(2l,1),$ or  $\frak{sp}(1,l-1)$, the $L_0$-representation $S^m(\frak{u}_{-1})$ is irreducible 
for all $m\geq 0$. 
\end{proposition} 
\begin{proof} Only the irreducibility of the $L_0$-module $S^m(\frak{u}_{-1})$ when $\frak{g}_0=\frak{so}(2l,1),\frak{sp}(1,l-1)$ needs to be established as the remaining assertions have already been established in \cite[\S4]{ow}.

When  $\frak{g}_0=\frak{so}(2l,1), L_0'\cong SU(l)$ and $\frak{u}_{-1}$, as an $L_0'$-representation, is isomorphic to the standard representation. Hence $S^m(\frak{u}_{-1})$ is irreducible as an $L_0'$-module---consequently as an $L_0$-module---for all $m$.  

When 
$\frak{g}_0=\frak{sp}(1,l-1)$, $L_0'=Sp(l-1)$.  Again 
$\frak{u}_{-1}$, as an $L_0'$-representation, is isomorphic to the standard representation of $Sp(l-1)$ (of dimension $2l-2$). 
Using the Weyl dimension formula, it follows that for any $m\geq 1$, $S^m(\frak{u}_{-1})$ is irreducible as $L_0'$-module and hence as an $L_0$-module.    
\end{proof}

\begin{remark} \label{qodd} {\em 
The centre $\bc H_{\nu^*}\subset \frak{l}$ acts via the character $-\nu^*/||\nu^*||^2=-||\epsilon||^2\epsilon^*/(4||\epsilon^*||^2)$ on the irreducible $\frak{l}$-representation $\frak{u}_{-1}$ and hence by $-k||\epsilon||^2\epsilon^*/(4||\epsilon^*||^2)$ on $S^k(\frak{u}_{-1})$ for all $k$.  
Suppose that $\mathcal{A}(\frak{u}_1,L)=\mathbb{C}[f]$ where $f\in S^k(\frak{u}_{-1})$ with  $\deg(f)=k>0$. Let $E_{q\epsilon^*}=\mathbb{C}f$ be the one-dimensional subrepresentation of $S^k(\frak{u}_{-1})$. Then $q=-k||\epsilon||^2/(4||\epsilon^*||^2)$. 

When $\frak{g}_0=\frak{sp}(p, l-p), 2\leq p\leq 2(l-p)$ with $p$ even, it turns out that $k=\deg(f)=p$ from \cite[\S4]{ow}. In this case $||\epsilon||^2=4$, $\epsilon^*=\nu^*$ and $||\epsilon^*||^2=p$.  Hence $q=-1$. 

When $\frak{g}_0=\frak{f}_{4,B_4}$, $k=\deg(f)=2$ from \cite[\S4]{ow}. In view of our normalization $||\nu||^2=2$, using \cite[Planche VIII]{bourbaki}, a straightforward calculation leads 
to $||\epsilon^*||^2=||\nu^*||^2=2, ||\epsilon||^2=4$ and so 
$q=-1$. }

{\it It follows from Remark \ref{spin} that when $Y$ does not admit a spin structure and $\mathcal{A}(\frak{u}_1,L)=\mathbb{C}[f]$, the value of $q$ is 
odd.}

{\em In fact it turns out that in all the remaining cases for which $\mathcal{A}(\frak{u}_1,L)=\mathbb{C}[f]$, the number $q$ is even.   In view of Remark \ref{spin}(i) we 
interpret this as follows: Denote by $\mathcal{K}_Y$ the canonical bundle of $Y$ and let $\mathbb{E}$ denote the 
line bundle over $Y$ determined by the $L_0$-representation $E:=\mathbb{C}f$.  Then the line bundle $\mathcal{K}_Y\otimes \mathbb{E}$ always admits a square root, that is, $\mathcal{K}_Y\otimes \mathbb{E}=\mathcal{L}\otimes \mathcal{L}$ for a (necessarily unique) line bundle $\mathcal{L}$ over $Y$. }
\end{remark}

%%%%%%%%%%%%%%%%%%%%%%%%%%%%%%%%%%%%%%%%%
%%%%%%%%%%%%%%%%%%%%%%%%%%%%%%%%%%%%%%%%%% 
\section{$L_0$-admissibility of the Borel-de Siebenthal 
discrete series}
%%%%%%%%%%%%%%

We begin by establishing the following proposition 
which implies that there is no loss of generality in 
confining our discussion throughout to the $K_0$-finite 
part of the Borel-de Siebenthal series rather than 
the discrete series itself when the $K_0$-finite part is $L_0$-admissible.
The following proposition is well known---see~\cite[Proposition~1.6]{kobinv98}.

Let $K_0$ be a maximal compact subgroup of a connected semisimple Lie group $G_0$ with finite centre and let $\pi$ be a unitary $K_0$-admissible representation of $G_0$ on a separable complex Hilbert space $\mathcal{H}$. Denote by $\mathcal{H}_{K_0}$ the $K_0$-finite vectors of $\mathcal{H}$ and by $\pi_{K_0}$ the restriction of $\pi$ to $\mathcal{H}_{K_0}$.  Thus $\mathcal{H}_{K_0}$ is dense in $\mathcal{H}$.

\begin{proposition}\label{kfinitepart}
Suppose that $\pi_{K_0}$   
is $L_0$-admissible where $L_0$ is a closed subgroup of $K_0$.   
Then any finite dimensional $L_0$-subrepresentation of 
$\pi$ is contained in $\mathcal{H}_{K_0}$.  In particular, $\pi$ is $L_0$-admissible. 
\end{proposition}
For a proof see \cite[Proposition 1.6]{kobinv98}.

%To see this, suppose that $v\in \mathcal{H}$ is contained in an irreducible (finite 
%dimensional) $L_0$-submodule of $\mathcal{H}$.  Then $\sum_{1\leq i\leq m}c_i\pi(x_i)v_0=v$ for some $L_0$-highest weight vector $v_0$ of weight, say, $\lambda$, for suitable $x_i\in L_0, c_i\in \bc$. Let $\{v_j\}$ be an orthonormal basis of $\mathcal{H}$ consisting of $L_0$-weight vectors, obtained by taking union of certain orthonormal bases of $L_0$-isotypic components of $\mathcal{H}_{K_0}.$  Write $v_0=\sum_{j}  a_j v_j$.  It is readily seen that $a_j$ is zero unless $v_j$ is an $L_0$-{\it highest} weight vector of weight $\lambda$.  This means that $v_0$ belongs to the $L_0$-isotypic component of $\pi_{K_0}$ having highest weight $\lambda$.  Since $\pi_{K_0}$ is $L_0$-admissible, it follows that $v_0\in \mathcal{H}_{K_0}$.  Hence $v\in \mathcal{H}_{K_0}$. \end{proof}

For the rest of this section we keep the notations of \S2. 
%Recall that $L_0=L_1\times L_2\subset K_0=K_1\times K_2$ where $L_2=K_2$ is semisimple and $L_1$ is the centralizer of a certain circle subgroup of $K_1$. Denote the derived group of $L_0$ by $L_0'$ so that  $L_0'=L_1'\times L_2$ where $L_1'$ is the derived subgroup of $L_1$. Let $Z(L_0)\cong \bs^1$ be the identity component of the centre of $L_1$.  
Any irreducible finite 
dimensional complex representation $E$ of $L_0=L_1\times L_2$ is isomorphic to a tensor product 
$E_1\otimes E_2$ where $E_j$ is an irreducible representation of $L_j, j=1,2$. In particular, if $E_1$ is one dimensional, then 
it is trivial as an $L_1'$ representation and $L_1$ acts on $E_1$ via a character $\chi:L_1/L_1'\lr \bs^1$.  If $E_2$ 
one dimensional, then it is trivial as an 
$L_2$-representation. 

Applying this observation to $S^k(\frak{u}_{-1})$ we see that one-dimensional $L_0$-subrepresentations of $S^{k}(\frak{u}_{-1})$ are all of the form $\bc h$ where $h\in S^k(\frak{u}_{-1})$ a weight vector which is invariant under the action of $L_1'\times L_2$. 
That is, $h$ is a relative invariant of $(\frak{u}_1,L)$. If $h\in S^k(\frak{u}_{-1})$ is a relative invariant, then so is $h^j$ for any $j\geq 1$. If $\chi = \sum_{\alpha\in \Delta_{-1}} r_\alpha \alpha, r_\alpha\geq 0$ is the weight of a relative invariant $h$, then, as $L_0'$ acts trivially on $\bc h$, we see that $\chi$ is a multiple of $\nu^*$.      

When $\frak{k}_1\cong\frak{su}(2)$ we have $L_1\cong \bs^1$.  Let $\pi$ be a representation of $G_0$ on a separable Hilbert space $\mathcal{H}$.  For example, 
$\pi$ is a Borel-de Siebenthal representation.  We have 
the following:

\begin{lemma}\label{qadmissibility}
Suppose that $\pi$ is $K_1$-admissible where $\frak{k}_1=\frak{su}(2)$.  
Then $\pi$ is $L_0$-admissible if and only if $\pi$ is 
$L_2$-admissible.   
\end{lemma}
\begin{proof}  
We need only prove that $L_0$ admissibility of $\pi$ implies the $L_2$ admissibility.   Note that $L_0'=L_2$. 
Assume that $\pi$ is not $L_2$ admissible.  Say $E$ 
is a $L_2$ type which occurs in 
$\pi$ with infinite multiplicity. In view of  Proposition 
\ref{kfinitepart} and since $L_0'=L_2$, the $L_2$-type $E$ actually occurs in $\pi_{K_0}$ with infinite multiplicity.  Then, denoting 
the irreducible $K_1$-representation of dimension $d+1$ by 
$U_d$, we deduce from $K_1$-admissibility of 
$\pi$ that the irreducible $K_0$-representations $U_{d_j}\otimes E$ occurs in $\pi$ where $(d_j)$ is a strictly increasing sequence of natural numbers. Without loss of generality we assume that all the $d_j$ are of same 
parity. Notice that $U_{c}$ as an $L_1$-module, is a submodule of $U_d$, if $c\leq d$ and $c\equiv d\mod 2$.   It follows that the $L_0$-type $U_{d_1}\otimes E$ occurs in {\it every}  summand of $\oplus_{j\geq 1}U_{d_j}\otimes E$.  Thus $\pi$ is not $L_0$-admissible.  
\end{proof}

{\it Proof of Proposition \ref{inadmissible}:} 
Let $h\in S^k(\frak{u}_{-1})$ be a relative invariant 
for $(\frak{u}_1,L)$ with weight $\chi=r\nu^*$. Denote by $\mathcal{L}$ the holomorphic line bundle $K_0\times_{L_0}\bc h \lr K_0/L_0=Y$.  Then $\mathcal{L}=\mathbb{E}_\chi$ and so $\mathbb{E}_\gamma\otimes \mathcal{L}^{\otimes j}=\mathbb{E}_{\gamma+j\chi}$ is a subbundle of 
the bundle $\mathbb{E}_\gamma\otimes \mathbb{S}^{jk}(\frak{u}_{-1})$ 
for all $j\geq 1$. Hence the $K_0$-module $H^s(Y;\mathbb{E}_{\gamma+j\chi})$ occurs in the Borel-de Siebenthal discrete series  $\pi_{\gamma+\rho_{\frak{g}}}$.  
The lowest weight of the $K_0$-module $H^s(Y;\mathbb{E}_{\gamma+j\chi})$ is  $w_{\frak{l}}^0(\gamma+j\chi+\rho_{\frak{k}})-w_{\frak{k}}^0\rho_{\frak{k}}=
w_{\frak{l}}^0(\gamma_0)+(t\nu^*+jr\nu^*)+\sum_{\alpha\in \Delta_2}\alpha$ where $\chi=r\nu^*$.  
As observed above, $\sum_{\alpha\in \Delta_2}\alpha=2s\nu^*/||\nu^*||^2$.  
Since $\nu^*$ is in the centre of $\frak{l}$, the irreducible $L_0'$ representation with lowest weight $w_{\frak{l}}^0(\gamma_0)$, namely $E_{\gamma_0}$,  occurs in $H^s(Y;\mathbb{E}_{\gamma+j\chi})$ for {\it all} $j\geq 1$. 
It follows that $\pi_{\gamma+\rho_{\frak{g}}}$ is not $L_0'$-admissible.   

It remains to prove the converse assuming $\frak{k}_1\cong \frak{su}(2)$.  We shall suppose that $\pi_{\gamma+\rho_{\frak{g}}}$ is not $L_0'$-admissible and that $S^m(\frak{u}_{-1})$ has no one-dimensional $L_0'$-submodules and arrive at a contradiction.
By Lemma \ref{qadmissibility}, $\pi_{\gamma+\rho_{\frak{g}}}$ is not $L_0$-admissible. By Proposition \ref{kfinitepart}, the $K_0$-finite part of $\pi_{\gamma+\rho_{\frak{g}}}$ is not 
$L_0$-admissible. 
 In view of  Proposition \ref{relativeinvariants} we have $\frak{g}_0=\frak{so}(4,1)$ or $\frak{sp}(1,l-1)$ and the $L_0$-module $S^m(\frak{u}_{-1})$ is irreducible for all $m$.  
The highest weight of $S^m(\frak{u}_{-1})$ as an $L_2$-module is $m (-\nu-a\nu^*)$ where $a\nu^*$ is the character by which $L_1=L_0/L_2\cong\bs^1$ acts on $\frak{u}_{-1}$.  

Now $H^1(\mathbb{P}^1; \mathbb{E}_\gamma\otimes\mathbb{S}^m(\frak{u}_{-1}))
=H^1(\mathbb{P}^1;\mathbb{E}_{(t+ma)\nu^*}\otimes \mathbb{E}_{-m\nu-ma\nu^*}\otimes \mathbb{E}_{\gamma_0})
=H^1(\mathbb{P}^1;\mathbb{E}_{(t+ma)\nu^*})\otimes E_{-m\nu-ma\nu^*}\otimes E_{\gamma_0}$ as a $K_1\times L_2$-module.  Since the $K_0$-finite part of $\pi_{\gamma+\rho_{\frak{g}}}$ is not $L_0$-admissible, there exist a $b$ and an $L_2$-dominant integral weight $\lambda$ such that the $L_0$-type  $E=E_{b\nu^*}\otimes E_\lambda$ occurs in $H^1(\mathbb{P}^1;\mathbb{E}_{(t+ma)\nu^*})\otimes E_{-m\nu-ma\nu^*}\otimes E_{\gamma_0}$ for infinitely many distinct values of $m$.  This implies that $E_{\lambda}$ occurs in $E_{-m\nu-ma\nu^*}\otimes E_{\gamma_0}$ for infinitely many values of $m$.
The highest weights of $L_2$-types occurring in $E_{-m\nu-ma\nu^*}\otimes E_{\gamma_0}$ are all of the form $-m\nu-ma\nu^*+\kappa_m$ where $\kappa_m$ is a weight of $E_{\gamma_0}$.  Thus 
$\lambda=-m\nu-ma\nu^*+\kappa_m$ for infinitely many $m$.  Since $E_{\gamma_0}$ is finite dimensional, it follows that for some weight $\kappa$ of $E_{\gamma_0}$, we have $\lambda-\kappa=-m\nu-ma\nu^*$ for infinitely many values of $m$, which is absurd. 
 \hfill $\Box$

\section{Holomorphic discrete series associated 
to a Borel-de Siebenthal discrete series}

We keep the notations of \S 2.  
Recall that $K_0/L_0$ is an irreducible compact Hermitian 
symmetric space.  Let $K_0^*$ be the dual of $K_0$ in $K$ with 
respect to $L_0$ so that $K_0^*/L_0$ is the non-compact 
irreducible Hermitian symmetric space dual to $K_0/L_0$. Note that  $\frak{k}=Lie(K_0^*)\otimes_{\mathbb{R}}\mathbb{C}$ 
and that $\frak{t}\subset \frak{l}$ is a Cartan subalgebra of $\frak{k}$. 
The sets of compact and non-compact roots of $(Lie(K_0^*),\frak{t}_0)$ are $\Delta_0$ and $\Delta_2\cup \Delta_{-2}$ respectively.  The unique non-compact simple root of $\Psi_{\frak{k}}$ is $\epsilon\in \Delta_2$.

Since the centralizer of $\bc H_{\nu^*}$ in $\frak{k}$ equals $\frak{l}$, the group $K_0^*$ admits holomorphic 
discrete series. 
See \cite[Theorem 6.6, Chapter VI]{knapp}.  The positive system 
$\Delta_{\frak{k}}^+$ is a Borel-de Siebenthal root 
order for $K_0^*$. 

Let $\gamma+\rho_{\frak{g}}$ be the Harish-Chandra parameter 
for a Borel-de Siebenthal discrete series of $G_0$. Thus 
$\gamma$ is the highest weight of an irreducible $L_0$-representation and $\langle \gamma+\rho_{\frak{g}},\beta\rangle <0$ for all $\beta\in \Delta_1\cup \Delta_2$.  
Clearly $\langle \gamma+\rho_{\frak{k}},\alpha\rangle >0$ 
for all positive compact roots $\alpha\in \Delta_0^+$.  We claim that $\langle \gamma+\rho_{\frak{k}},\beta \rangle <0$ for all positive non-compact roots $\beta\in \Delta_2$. To see this, let $\beta_i\in \Delta_i, i=1,2$.  Observe that $\beta_1+\beta_2$ is not 
a root and so  
$\langle \beta_1,\beta_2\rangle\geq 0$.  It follows that $\langle \rho_{\frak{k}},\beta_2\rangle =\langle \rho_{\frak{g}}- 1/2\sum_{\beta_1\in \Delta_1}\beta_1,\beta_2\rangle 
=\langle \rho_{\frak{g}},\beta_2\rangle - 1/2\sum_{\beta_1\in \Delta_1}\langle \beta_1,\beta_2\rangle\leq \langle \rho_{\frak{g}},\beta_2\rangle$.  So  $\langle\gamma+\rho_{\frak{k}},\beta\rangle\leq \langle \gamma+\rho_{\frak{g}},\beta \rangle<0$ for all $\beta\in \Delta_2$.  Thus, by \cite[Theorem 6.6, Ch. VI]{knapp}, $\gamma+\rho_{\frak{k}}$ is the Harish-Chandra parameter for a holomorphic discrete series $\pi_{\gamma+\rho_{\frak{k}}}$ of $K_0^*$, which is naturally associated to the Borel-de Siebenthal discrete series $\pi_{\gamma+\rho_{\frak{g}}}$ of $G_0$.  

The $L_0$-finite part of $\pi_{\gamma+\rho_{\frak{k}}}$ equals 
$E_\gamma\otimes S^*(\frak{u}_{-2})$, where $E_\gamma$ is the 
irreducible $L_0$-representation with highest weight 
$\gamma$. Write $\gamma=\lambda+\kappa$ where $\lambda$ and $\kappa$ are dominant weights of $\frak{l}_1^\bc$ and $\frak{l}_2^\bc$ respectively. 
We have $E_\gamma=E_{\lambda}\otimes E_{\kappa}$. 
Hence $(\pi_{\gamma+\rho_{\frak{k}}})_{L_0}
=E_{\kappa}\otimes(E_{\lambda}\otimes S^*(\frak{u}_{-2}))=E_{\kappa}\otimes (\pi_{\lambda+\rho_{\frak{k}_1^\bc}})_{L_1}$, 
where $\pi_{\lambda+\rho_{\frak{k}_1^\bc}}$ 
is the holomorphic discrete series of $K_1^*$ with 
Harish-Chandra parameter $\lambda+\rho_{\frak{k}_1^\bc}$.  Here $K_1^*$ is 
the Lie subgroup of $K_0^*$ dual to $K_1$.

%%%%%%%%%%%%%%%%%%%%%%%%%%%%%%%%%%%%%%%%%%%%%%%%%%%

\section{Common $L_0$-types in the quaternionic case}
We now focus on the quaternionic case, namely,  
when $Lie(K_1)=\frak{su}(2)$. This case is characterized by 
the property that $-\mu$ is connected to $\nu$ in the 
extended Dynkin diagram of $\frak{g}$. 
 In this case $\Delta_2
=\{\mu\}, 
L_1\cong\mathbb{S}^1, Y=\mathbb{P}^1, L_2=[L_0,L_0]$, and, $\frak{l}'=[\frak{l},\frak{l}]=\frak{l}_2^\bc$.  Also, since both $\mu$ and $\nu^*$ are orthogonal to  
$\frak{l}_2^\bc$,  $\mu$ is a non-zero multiple of $\nu^*$.   Write $\mu=d\nu^*$.  
Since $\mu=2\nu+\beta$ where $\beta$ is a linear combinations of  roots of $\frak{l}_2^\bc$, we obtain  $||\mu||^2=d\langle\nu^*,\mu\rangle=d\langle\nu^*,2\nu\rangle=d||\nu||^2=2d$ as $||\nu||^2=2$.  Since $s_\nu(\mu)=\mu-d\nu$ is a root and since 
$\mu-3\nu$ is not a root, we must have $d=1$ or $2$.  For example, when $\frak{g}_0=\frak{so}(4,2l-3)$ or the split real form of the exceptional Lie algebra $\frak{g}_2$, we have $d=1$, whereas when $\frak{g}_0=\frak{sp}(1,l-1)$, we have $d=2$. 

Clearly $\frak{k}_1^\bc=
\frak{g}_\mu\oplus\bc H_\mu\oplus\frak{g}_{-\mu}\cong\frak{sl}(2,\bc) $.  The fundamental weight of $\frak{k}_1^\bc$ equals $\mu^*:=\mu/2=d\nu^*/2$. 
We shall denote by $U_k$ the $(k+1)$-dimensional $\frak{k}_1^\bc$-module with highest weight $k\mu^*=dk\nu^*/2$.  
Also, $\mathbb{C}_\chi$ denotes the one dimenional $\frak{l}_1^\bc$-module corresponding to a character $\chi\in 
\mathbb{C}\nu^*$. 

Let $\gamma=\gamma_0+t \nu^*$ where $\gamma_0$ is a dominant integral weight of $\frak{l}'=\frak{l}_2^\bc$ and $t$  satisfies the `sufficiently negative' condition (1).   We have the following lemma.  

\begin{lemma} \label{negativity}
Suppose that $\frak{k}_1=\frak{su}(2)$, $\gamma=\gamma_0+t\nu^*$ where $\gamma_0$ is an $\frak{l}'$-dominant weight.  Then $t$ satisfies the `sufficient negativity' condition {\em (1)} if and only 
if the following inequalities hold:\\
\[t<-\frac{d}{4}(|\Delta_1|+2), \textrm{~and}~~
t<-\langle\gamma_0,w^0_{\frak{l}}(\nu)\rangle-(1/2)(\sum a_i||\psi_i||^2)\]
where $w_{\frak{l}}^0(\nu)=\sum a_i\psi_i$ is the highest root in $\Delta_1$.  
\end{lemma}
\begin{proof}
Since $\gamma_0$ is a dominant integral weight of  $\frak{l}'=\frak{l}_2^\bc$ and since $\mu=d\nu^*$ is 
orthogonal to $\frak{l}_2^\bc$, we have $\langle \gamma_0,\mu\rangle=0$.  Since $\rho_{\frak{g}}=(1/2)\sum_{\alpha\in \Delta^+} \alpha$, we get $\langle\rho_{\frak{g}},\mu\rangle=(d/2)
(\sum_{\alpha\in \Delta_0^+}\langle \alpha,\nu^*\rangle
+\sum_{\alpha\in \Delta_1}\langle \alpha,\nu^*\rangle
+\sum_{\alpha\in \Delta_2}\langle\alpha,\nu^*\rangle)
=(d/2)(|\Delta_1|+2|\Delta_2|) $, since $\langle\alpha,\nu^*\rangle=i\langle \nu,\nu^*\rangle= i$ whenever $\alpha\in \Delta_i, i=0,1,2$. 
Since $|\Delta_2|=1$, we have $t<-(1/2)\langle\gamma_0+\rho_{\frak{g}},\mu\rangle$ 
if and only if $t<-(d/4)(|\Delta_1|+2)$.

Now $w^0_{\frak{l}}(\nu)=\sum a_j\psi_j$ is the highest weight of $\frak{u}_1$, which is indeed the highest root 
in $\Delta_1$.   Therefore $\langle\rho_{\frak{g}},w^0_{\frak{l}}(\nu)\rangle=\langle\sum\psi_i^*,\sum a_j\psi_j\rangle=(1/2)(\sum a_i||\psi_i||^2)$.  This completes the proof.
\end{proof}

\noindent
{\it Proof of Theorem \ref{main}:}  
Write $\frak{u}_{-1}=E_1\otimes E_2$ where $E_i$ is an irreducible $L_i$-module.  By our hypothesis $L_1\cong \bs^1=\{\exp(i\lambda H_\mu)|\lambda\in \br\}$ and so $E_1$ is $1$-dimensional, given by the character $-\nu^*/||\nu^*||^2=-\mu^*$.  On the other hand, the highest weight of $E_2$ is $-(\nu-\mu^*)$.  Hence $E_2\cong E_{\mu^*-\nu}$.  
Since $E_1$ is one dimensional, we have $S^m(\frak{u}_{-1})=\bc_{-m\mu^*}\otimes S^m(E_{\mu^*-\nu})$.  
On the other hand $\frak{u}_{-2}$ is $1$-dimensional and is isomorphic as an $L_0$-module to $\bc_{-\mu}=\bc_{-2\mu^*}$.  Therefore  
$S^m(\frak{u}_{-2})=\bc_{-2m\mu^*}$.   

The vector bundle $\mathbb{E}$ over $Y=K_1/L_1$ associated to any $L_2$ representation space $E$ is clearly isomorphic to the product bundle $Y\times E\lr Y$. 
Therefore the bundle $\mathbb{E}_\gamma\otimes \mathbb{S}^m(\frak{u}_{-1})$ over $Y=\mathbb{P}^1$ is 
isomorphic to $\mathbb{E}_{(2t/d-m)\mu^*}\otimes 
E_{\gamma_0}\otimes S^m(E_{\mu^*-\nu})$.  It follows that $H^1(Y; \mathbb{E}_\gamma\otimes \mathbb{S}^m(\frak{u}_{-1}))\cong H^1(Y;\mathbb{E}_{(2t/d-m)\mu^*})\otimes E_{\gamma_0}\otimes S^m(E_{\mu^*-\nu})
\cong U_{-2t/d+m-2}\otimes E_{\gamma_0}\otimes S^m(E_{\mu^*-\nu})$.  By Theorem \ref{bds} 
we conclude that 
\[(\pi_{\gamma+\rho_\frak{g}})_{K_0}=\bigoplus_{m\geq 0}
U_{(m-2t/d-2)}\otimes E_{\gamma_0}\otimes S^m(E_{\mu^*-\nu}).\eqno{(2)}\]

We now turn to the description of the holomorphic 
discrete series $\pi_{\gamma+\rho_{\frak{k}}}$ of $K_0^*=K_1^*  K_2$.
Recall from \cite{repka} the following description of the holomorphic discrete series of $K_1^*$ determined by $t\nu^*=(2t/d)\mu^*$, namely, 
$(\pi_{(2t/d)\mu^*+\rho_{\frak{k}_1^\bc}})_{L_1}=\oplus_{r\geq 0} \bc_{(2t/d)\mu*}\otimes S^r(\frak{u}_{-2})=\oplus_{r\geq 0}\bc_{(2t/d-2r)\mu^*}.$  
It follows that the holomorphic discrete series of $K_0^*$ 
determined by $\gamma$ is  
\[(\pi_{\gamma+\rho_\frak{k}})_{L_0}=\bigoplus_{r\geq 0}\bc_{(2t/d-2r)\mu^*}\otimes E_{\gamma_0}.\eqno{(3)}\]

Comparing (2) and (3) we observe that there exists an $L_0$-type common to $(\pi_{\gamma+\rho_\frak{g}})_{K_0}$ and $\pi_{\gamma+\rho_\frak{k}}$ {\it  if and only if} 
the following two conditions hold:\\
(a) $E_{\gamma_0}$ occurs in $E_{\gamma_0}\otimes S^m(E_{\mu^*-\nu})$. \\
(b) Assuming that (a) holds for some $m\geq 0$,  
$(2t/d-2r)\mu^*$ occurs as a weight in $U_{m-2t/d-2}$ for some $r$, that is, $2t/d-2r=(m-2t/d-2)-2i$ for some $0\leq i \leq (m-2t/d -2)$.

    First suppose that $\frak{g_0} = \frak{so}(4,1)$ or $ \frak{sp}(1,l-1), l>1$. In view of Proposition \ref{inadmissible} and Proposition \ref{kfinitepart}, the Borel-de Siebenthal discrete series $\pi_{\gamma+\rho_\frak{g}}$ is $L_0$-admissible and any $L_0$-type in $\pi_{\gamma+\rho_\frak{g}}$ is contained in $(\pi_{\gamma+\rho_\frak{g}})_{K_0}$. Also $S^m(E_{\mu^* - \nu})$ is irreducible with highest weight $m(\mu^* - \nu)$ (see Proposition \ref{relativeinvariants}). Recall that the highest weights of  irreducible sub representations which occur in a tensor product $E_\lambda\otimes E_\kappa$ of two irreducible representations of $\frak{l}^\bc_2$ are all of the form $\theta+\kappa$ where $\theta$ is a weight of $E_\lambda.$  So if (a) holds, then $\gamma_0=m(\mu^*-\nu)+\theta$, for some weight $\theta$ of $E_{\gamma_0}$. This implies $\gamma_0-\theta=m(\mu^*-\nu)$, which holds for atmost finitely many $m$ since the number of weights of $E_{\gamma_0}$ is finite. So by (a), there are atmost finitely many $L_0$-types common to $\pi_{\gamma+\rho_\frak{g}}$ and $\pi_{\gamma+\rho_\frak{k}}$.

Moreover, if $\gamma_0=0$, then the trivial $L_0$-representation $E_{\gamma_0}$ occurs in $E_{\gamma_0}\otimes S^m(E_{\mu^*-\nu})=E_{m(\mu^*-\nu)}$ 
only when $m=0$.  Since $2t/d-2r \leq 2t/d <2t/d+2$ for all $r\geq 0$,  $(2t/d-2r)\mu^*$ cannot be a weight of $U_{-2t/d-2}$ for all $r\geq 0$. So in view of (a) and (b), there are no common $L_0$-types between $\pi_{\gamma+\rho_\frak{g}}$ and $\pi_{\gamma+\rho_\frak{k}}$.

Now suppose that $\frak{g_0}\neq \frak{so}(4,1), \frak{sp}(1,l-1) , l>1$. In view of Proposition \ref{relativeinvariants}, we see that $\mathcal{A}(\frak{u}_1,L)=\bc[f]$, where $f$ is a relative invariant 
(hence is a homogeneous polynomial) of positive degree, say of degree $k$. Then the trivial module is a sub module of the $L_0$-module $S^{jk}(E_{\mu^*-\nu})$ for all $j\geq 0$. So $E_{\gamma_0}$ occurs in $E_{\gamma_0}\otimes S^{jk}(E_{\mu^*-\nu})$ for all $j\geq 0$. That is (a) holds. \\
 Let $r$ be a non negative integer. Then $(2t/d-2r)\mu^*$ is a weight of 
 $U_{jk-2t/d-2}$ for some $j\geq 0$ if and only if $2t/d-2r=(jk-2t/d-2)-2i$ for some $0\leq i\leq (jk-2t/d-2)$ if and only if $jk$ is even and 
 $jk\geq 2(r+1)$.

So in view of (a) and (b), each $L_0$-type in $\pi_{\gamma +\rho_\frak{k}}$ occurs in $\pi_{\gamma +\rho_\frak{g}}$ with infinite multiplicity. This completes the proof. \hfill$\Box$

%%%%%%%%%%%%%%%%%%%%%%%%%%%%%%%%%%%%%%%%%%%%%%%%%%

\section{Decomposition of the symmetric algebra of the isotropy representation}
Let $(K_0,L_0)$ be a  Hermitian symmetric pair of compact type where $K_0$ is simply connected and {\it simple.}  Fix a 
maximal torus $T_0\subset L_0$. 
In this section we recall the description of the 
decomposition of the symmetric powers of the 
isotropy representation of $L_0$ (on the tangent space 
at the identity coset $o\in K_0/L_0=:Y$).  
Let $K_0^*$ denote the dual of $K_0$ with respect to $L_0$. We shall denote the maximal compact subgroup 
of $K_0^*$ corresponding to $Lie(L_0)$ by the same 
symbol $L_0$.
%We shall 
%also consider the (non-compact) 
Thus $(K^*_0,L_0)$ is the non-compact dual of $(K_0,L_0)$ and $X:=K_0^*/L_0$ is the non-compact Hermitian symmetric space dual to $Y$. 

To conform to the notations of \S2, we shall 
denote the set of roots of $\frak{k}=\frak{k}_0^\bc$ with respect to the Cartan subalgebra $\frak{t}=\frak{t}_0^\bc$ by $\Delta_{\frak{k}}$, the set of positive (respectively negative) non-compact roots of a Borel-de Siebenthal positive system of $K_0^*$ 
by $\Delta_2$ (respectively $\Delta_{-2}$) and the holomorphic tangent space 
at $o$ by $\frak{u}_{2}=\sum_{\alpha\in \Delta_{2}}\bc X_{\alpha}$, which affords the 
isotropy representation. The highest weight of the cotangent space $\frak{u}_{-2}$ at $o$ is  
$-\epsilon$, where $\epsilon$ is the simple non-compact root of $K^*_0$. 

Recall, from \cite[Ch. VIII]{helgason}, that two roots $\alpha,\beta\in \Delta_{-2}$ are 
called {\it strongly orthogonal} if $\alpha+\beta, \alpha-\beta$ are not roots of $\frak{k}$.  Since sum of 
two non-compact positive roots is never a root and their 
difference is, if at all, a compact root;  
$\alpha, \beta\in \Delta_{-2}$ are strongly orthogonal if and only if they are orthogonal, that is, $\langle \alpha,\beta\rangle =0$.  Let $\Gamma\subset \Delta_{-2}$ be a maximal set of 
strongly orthogonal roots.  The cardinality of $\Gamma$  
equals the rank of $X$, that is, the maximum  dimension of a Euclidean space that can be imbedded 
in $X$ as a totally geodesic submanifold. 

\subsection{}\label{stronglyorthogonalroots} We now consider a specific maximal set $\Gamma\subset \Delta_{-2}$ of strongly orthogonal roots 
whose elements $\gamma_1,\ldots,\gamma_r$ are 
inductively defined as follows:  this notation should not be confused with the notation~$\gamma_1$, $\gamma_2$ used in~\S\ref{ss:lnot}.  Fix an ordering of the 
simple roots and consider the induced lexicographic ordering on $\Delta_{\frak{k}}$.  Now let $\gamma_1:=-\epsilon$, the highest root in $\Delta_{-2}.$  Having defined $\gamma_1,\ldots, \gamma_i$, let $\gamma_{i+1}$ be the 
highest root in $\Delta_{-2}$ which is orthogonal to 
$\gamma_j, 1\leq j\leq i$.  

Denote by $E_\gamma$ the irreducible 
$L_0$-representation with highest weight $\gamma$. 
We have the following decomposition theorem~\cite{schmid}, which is a far reaching generalization of 
the fact that the symmetric power of the defining representation of the special unitary group is irreducible. See \cite[Theorem 10.25]{knapp2}.

\begin{theorem}\label{schmid}\textup{(see~\cite{schmid})} With the above notations, 
one has the decomposition $S^m(\frak{u}_{-2})$ as an $L_0$-representation 
\[S^m(\frak{u}_{-2})
=\bigoplus E_{a_1\gamma_1+\cdots+a_r\gamma_r}\] 
where the sum is over all partitions $a_1\geq \cdots \geq a_r\geq 0$ of $m$. \hfill $\Box$
\end{theorem}

Let $\epsilon^*$ be the fundamental weight corresponding to $\epsilon$ and $\frak{z}_{\frak{l}}^*$ be the dual space of  $\frak{z}_{\frak{l}}$. Note that  $\frak{z}_{\frak{l}}^*=\bc \epsilon^*$. Hence $E_\gamma$ is one dimensional precisely when  $\gamma=k\epsilon^*$ for some integer $k$. Now we see from the above theorem that 
$S^m(\frak{u}_{-2})$ admits a $1$-dimensional $L_0$-subrepresentation precisely when 
there exists non negative integers $a_1\geq \cdots \geq 
a_r\geq 0$ such that $\sum a_i\gamma_i=c_0\epsilon^*$ for 
some constant $c_0$. The first part of the following proposition gives a criterion for this to happen.

\begin{proposition}\label{sumgamma}
(i)  Let $\Gamma=\{\gamma_1, \ldots, \gamma_r\}$ be the maximal set of strongly orthogonal roots obtained as above. 
 Let $w^0_{\frak{k}}$ denote the longest element of the Weyl group 
 of $(\frak{k},\frak{t})$. Suppose that $w^0_{\frak{k}}(-\epsilon)=\epsilon.$  Then $\sum_{1\leq i\leq r} \gamma_i=-2\epsilon^*$.  Conversely, if $\sum_{1\leq i\leq r} a_i\gamma_i$ is a non-zero multiple of $\epsilon^*$ where $a_i\in \bz$, then $a_i=a_j~\forall 1\leq i,j\leq r,$ and,  
$w^0_{\frak{k}}(\epsilon)=-\epsilon$.
 
 (ii) Moreover, for any $1\leq j\leq r$, if the coefficient of a compact simple root $\alpha$ of $\frak{k}$ in the expression of $\sum_{1\leq i\leq j} \gamma_i$ is non-zero, then $\sum_{1\leq i\leq j} \gamma_i$ is orthogonal to $\alpha$ (without any assumption on $w^0_{\frak{k}}$).
\end{proposition}

\begin{proof} Our proof involves a straightforward verification using 
the classification of irreducible Hermitian symmetric pairs of non-compact type. See \cite[\S 6, Ch. X]{helgason}. We follow the labelling conventions of Bourbaki \cite[Planches I-VII]{bourbaki} and make use of the description of the root system, especially in cases E-III and E-VII. Note that $-w^0_{\frak{k}}$ induces an automorphism of the Dynkin diagram of $\frak{k}$. In particular, $-w^0_{\frak{k}}(\epsilon)=\epsilon$ when the Dynkin diagram of $K_0$ admits no symmetries.

\noindent
{\it Case A III}:  $(\frak{k}_0^*,\frak{l}_0)=(\frak{su}(p,q),\frak{s}(\frak{u}(p)\times \frak{u}(q))), p\leq q$. The simple roots are $\psi_i=\e_i -\e_{i+1}$, $1\leq i\leq p+q-1$. If $p+q>2$, then $-w^0_{\frak{k}}$ induces the order $2$ automorphism of the Dynkin 
diagram of $\frak{k}$, which is of type $A_{p+q-1}$.  Thus $-w^0_{\frak{k}}(\psi_j)=\psi_{p+q-j}$ in any case.  The simple non-compact root is $\epsilon=\psi_p=\e_{p}-\e_{p+1}$, all other simple roots are compact roots.  Therefore $-w^0_{\frak{k}}(\psi_p)=\psi_p$ if and only if $p=q$.  
On the other hand, the set of 
negative non-compact roots $\Delta_{-2}= \{\e_j-\e_i\mid 
1\leq i\leq p<j\leq p+q\}$ and $\Gamma=\{\gamma_j:=\e_{p+j}-\e_{p-j+1}\mid 1\leq j\leq p\}$.  If $p=q$, then $\sum_{1\leq j\leq p} \gamma_j = \sum_{1\leq j\leq q}\e_{p+j}-\sum_{1\leq j\leq p}\e_{p-j+1}$.  Using the fact that $\sum_{1\leq i\leq p+q}\e_i=0$, we see that $\sum_{1\leq j\leq p}\gamma_j=-2(\sum_{1\leq j\leq p}\e_j)=-2\epsilon^*$ if $p=q$.  

For the converse part, assume that $\sum_j a_j\gamma_j=m\epsilon^*, m\neq 0$.  It is evident when $p<q$ that $\sum a_j\gamma_j$ is not a multiple of $\epsilon^*$ (since $\e_{p+q}$ does not occur in the sum).    Since the $\gamma_j,1\leq j\leq p$, 
are linearly independent, the uniqueness of the expression of $\epsilon^*$ as a linear combination of the $\gamma_j$ implies that $a_j=a_1$ for all $j$.

 To prove $(ii)$, note that $\gamma_1=-\epsilon$ and $\gamma_j=-(\epsilon+\psi_{p-j+1} +\cdots+\psi_{p-1}+\psi_{p+1}+\cdots+\psi_{p+j-1})$, $2\leq j\leq p$. So the only compact simple roots whose coefficients are non-zero in the expression of $\sum_{1\leq i\leq j} \gamma_i(j>1)$ are $\psi_i$ ($p-j+1\leq i\leq p+j-1$, $i\ne p$). Note that  $\sum_{1\leq i\leq j} \gamma_i=-(\e_{p-j+1}+\cdots+\e_p-\e_{p+1}-\cdots-\e_{p+j})$. Hence $\langle\sum_{1\leq i\leq j} \gamma_i, \psi_i\rangle=0$ for all $p-j+1\leq i\leq p+j-1$, $i\ne p$.

\noindent
{\it Case D III}: $(\frak{so}^*(2p),\frak{u}(p)), p\geq 4$. The simple roots are $\psi_i=\e_i -\e_{i+1}$, $1\leq i\leq p-1$ and $\psi_p=\e_{p-1}+\e_{p}$. 
In this case the only non-compact simple root  
$\epsilon=\psi_p=\e_{p-1}+\e_{p}$; $\epsilon^*=(1/2)(\sum_{1\leq j\leq p}\e_j)$. The set of non-compact positive roots is 
$\Delta_2=\{\e_i+\e_j\mid 1\leq i<j\leq p\}$ and $\Gamma=\{\gamma_j=-(\e_{p-2j+1}+\e_{p-2j+2})\mid 1\leq j\leq \lfloor p/2\rfloor\}$. 
So $\sum_{1\leq j\leq \lfloor p/2\rfloor}\gamma_j=-2\epsilon^*$ if $p$ is even. 
On the other hand $w^0_{\frak{k}}$ maps $\epsilon$ to $-\epsilon$ precisely when $p$ is even. 

When $p$ is odd, it is readily seen that $\sum_j a_j\gamma_j$ is not a non-zero multiple of $\epsilon^*$ since $\e_1$ does not occur in the sum.

 To prove $(ii)$, note that $\gamma_1=-\epsilon$ and $\gamma_j=-(\epsilon+\psi_{p-2j+1} +2\psi_{p-2j+2}+\cdots+2\psi_{p-2}+\psi_{p-1})$, $2\leq j\leq \lfloor p/2\rfloor$. So the only compact simple roots whose coefficients are non-zero in the expression of $\sum_{1\leq i\leq j} \gamma_i(j>1)$ are $\psi_i$ ($p-2j+1\leq i\leq p-1$). Note that  $\sum_{1\leq i\leq j} \gamma_i=-(\e_{p-2j+1}+\cdots+\e_p)$. Hence $\langle\sum_{1\leq i\leq j} \gamma_i, \psi_i\rangle=0$ for all $p-2j+1\leq i\leq p-1$.

\noindent
{\it Case BD I (rank$=2$):} 
$(\frak{so}(2,p),\frak{so}(2)\times \frak{so}(p)), p>2$. 
We have $\epsilon=\psi_1=\e_1-\e_2, \epsilon^*=\e_1$ and $w^0_{\frak{k}}(\epsilon)=-\epsilon$.  
Now $\Delta_2=\{\e_1\pm \e_j\mid 2\leq j\leq p\}\cup\{\e_1\}$ if $p$ is odd and is equal to $\{\e_1\pm \e_j\mid 
2\leq j\leq p\}$ if $p$ is even.  For any $p$,   
$\Gamma=\{\gamma_1=-(\e_1-\e_2), \gamma_2=-(\e_1+\e_2)\}$. Clearly $a_1\gamma_1+a_2\gamma_2=
m\epsilon^*$ if and only if $a_1=a_2$. 
Since in this case rank is $2$ and $\gamma_1+\gamma_2= -2\epsilon^*$, $(ii)$ is obvious.

\noindent 
{\it Case C I:}
$(\frak{sp}(p,\mathbb{R}),\frak{u}(p)), p\geq 3$.  The simple roots are $\psi_i=\e_i -\e_{i+1}$, $1\leq i\leq p-1$ and $\psi_p=2\e_{p}$. 
We have $\epsilon=2\e_p, \epsilon^*=\sum_{1\leq j\leq p}\e_j$, 
and $w^0_{\frak{k}}(\epsilon)=-\epsilon$. Also $\Delta_2=\{\e_i+\e_j\mid 1\leq i\leq j\leq p\}$. Therefore $\Gamma=\{\gamma_j:=-2\e_{p-j+1}\mid 1\leq j\leq p\}$. Evidently 
$\sum_{1\leq j\leq p}\gamma_j=-2\epsilon^*$. 

The converse part is obvious in this case.

 To prove $(ii)$, note that $\gamma_1=-\epsilon$ and $\gamma_j=-(\epsilon+2\psi_{p-j+1} +\cdots+2\psi_{p-1})$, $2\leq j\leq p$. So the only compact simple roots whose coefficients are non-zero in the expression of $\sum_{1\leq i\leq j} \gamma_i(j>1)$ are $\psi_i$ ($p-j+1\leq i\leq p-1$). Note that  $\sum_{1\leq i\leq j} \gamma_i=-2(\e_{p-j+1}+\cdots+\e_p)$. Hence $\langle\sum_{1\leq i\leq j} \gamma_i, \psi_i\rangle=0$ for all $p-j+1\leq i\leq p-1$. 

\noindent
{\it Case E III:} 
$(\frak{e}_{6, -14}, \frak{so}(10)\oplus \frak{so}(2))$. The simple roots are $\psi_1=(1/2)(\e_8-\e_6-\e_7+\e_1-\e_2-\e_3-\e_4-\e_5), \psi_2=\e_1+\e_2, \psi_3=\e_2-\e_1, \psi_4=\e_3-\e_2, \psi_5=\e_4-\e_3,
\psi_6=\e_5-\e_4$. In this case 
the rank is $2$, $\epsilon=\psi_1=(1/2)(\e_8-\e_6-\e_7+\e_1-\e_2-\e_3-\e_4-\e_5), $ and 
$\epsilon^*= (2/3)(\e_8-\e_7-\e_6)$.  We have $-w^0_{\frak{k}}(\epsilon)=\psi_6\neq \epsilon$.  
Now  $\Delta_2=\{(1/2)(\e_8-\e_7-\e_6+\sum_{1\leq i\leq 5}(-1)^{s(i)}\e_i)\mid 
s(i)=0,1, \sum_i s(i)\equiv 0\mod 2 \}.$ 
There are five roots in $\Delta_{-2}$ which are   
orthogonal to $\gamma_1=-\epsilon.$  Among these 
the highest is $\gamma_2=-(1/2)(\e_8-\e_6-\e_7-\e_1+\e_2+\e_3
+\e_4-\e_5)$.  Thus $\Gamma=\{\gamma_1,\gamma_2\}$. 
Now $a_1\gamma_1+a_2\gamma_2$ is not a multiple of $\epsilon^*$  for any $a_1,a_2\geq 0$ unless $a_1=a_2=0$. 

 Note that $\gamma_2=-(\epsilon+\psi_2+2\psi_3+2\psi_4+\psi_5)$, $\gamma_1 +\gamma_2=-(\e_8-\e_7-\e_6-\e_5)$. Hence $\langle\gamma_1+\gamma_2, \psi_i\rangle=0$ for all $2\leq i\leq 5$.

\noindent
{\it Case E VII:} 
$(\frak{e}_{7, -25},\frak{e}_6\oplus \frak{so}(2))$.  The simple roots are $\psi_1=(1/2)(\e_8-\e_6-\e_7+\e_1-\e_2-\e_3-\e_4-\e_5), \psi_2=\e_1+\e_2, \psi_3=\e_2-\e_1, \psi_4=\e_3-\e_2, \psi_5=\e_4-\e_3,
\psi_6=\e_5-\e_4,\psi_7=\e_6-\e_5$. In this case rank$=3$, $\epsilon=\psi_7=\e_6-\e_5, \epsilon^*=\e_6+(1/2)(\e_8-\e_7), w^0_{\frak{k}}(-\epsilon)=\epsilon $. 
$\Delta_2=\{\e_6-\e_j,\e_6+\e_j, 1\leq j\leq 5\}\cup\{\e_8-\e_7\}
\cup\{(1/2)(\e_8-\e_7+\e_6+\sum_{1\leq j\leq 5}(-1)^{s(j)}\e_j)\mid 
s(j)=0,1, \sum_j s(j)\equiv 1\mod 2\}$.
Now $\Gamma=\{\gamma_1=\e_5-\e_6, \gamma_2=-\e_5-\e_6,\gamma_3=\e_7-\e_8\}$ and we have $\gamma_1+\gamma_2+\gamma_3=-2\epsilon^*$. 
The converse part is easily established.

We have  $\gamma_2=-(\epsilon+\psi_2+\psi_3+2\psi_4+2\psi_5+2\psi_6)$, $\gamma_1 +\gamma_2=-2\e_6$.  Hence $\langle\gamma_1+\gamma_2, \psi_i\rangle=0$ for all $2\leq i\leq 6$. Also $\gamma_1+\gamma_2+\gamma_3= -2\epsilon^*$. So $(ii)$ is proved.
\end{proof}

As a corollory we obtain the following.

\begin{proposition}\label{holadmissibility}
Suppose that $K^*_0/L_0$ is an irreducible Hermitian symmetric space of non-compact type and let $\pi_{\gamma+\rho_{\frak{k}}}$ be 
a holomorphic discrete series of $K^*_0$.  
If $w^0_{\frak{k}}(\epsilon)=-\epsilon$, then $(\pi_{\gamma+\rho_{\frak{k}}})_{L_0}$ is not $L_0'$-admissible.  
Conversely, if a holomorphic discrete series 
$\pi_{\gamma+\rho_{\frak{k}}}$ of $K_0^*$ is not $L_0'$-admissible, then $w_{\frak{k}}^0(\epsilon)=-\epsilon$. 
\end{proposition} 
\begin{proof}
One has the following description of $(\pi_{\gamma+\rho_{\frak{k}}})_{L_0}$ due to Harish-Chandra: 
$(\pi_{\gamma+\rho_{\frak{k}}})_{L_0}=\oplus_{m\geq 0} E_\gamma \otimes S^m(\frak{u}_{-2})$.  Suppose that 
$w^0_{\frak{k}}(\epsilon)=-\epsilon$.  Then by 
Proposition \ref{sumgamma} and 
Theorem~\ref{schmid} we see that $E_\gamma
\otimes E_{-a\epsilon^*}$ occurs in $(\pi_{\gamma+\rho_{\frak{k}}})_{L_0}$ for infinitely many values of $a$. Since 
$E_{-\epsilon^*}$ is one-dimensional, it is trivial as 
an $L_0'$-representation.  Hence $(\pi_{\gamma+\rho_{\frak{k}}})_{L_0}$ is not $L_0'$-admissible. 

 Conversely, since $\pi_{\gamma+\rho_{\frak{k}}}$ is not $L_0'$-admissible, in view of Proposition \ref{kfinitepart} we have, $(\pi_{\gamma+\rho_{\frak{k}}})_{L_0}$ is not $L_0'$-admissible. Suppose that $w^0_{\frak{k}}(-\epsilon)\neq \epsilon$.
Any $L_0'$-type in  $(\pi_{\gamma+\rho_{\frak{k}}})_{L_0}$ is of the form $E_{\sum a_j\gamma_j +\kappa}$ (considered as $L_0'$-module) for some weight $\kappa$ of $E_\gamma$. Since the set of weights of $E_{\gamma}$ is finite,  $(\pi_{\gamma+\rho_{\frak{k}}})_{L_0}$ is not $L_0'$ admissible implies  $S^*(\frak{u}_{-2})$ is not $L_0'$ admissible.  If $E_{\sum a_j\gamma_j}\cong E_{\sum b_j\gamma_j}$ as $L_0'$-modules, then $\sum (a_j-b_j)\gamma_j$ is a multiple 
of $\epsilon^*$.  Proposition \ref{sumgamma} implies 
that $a_j=b_j, 1\leq j\leq r.$ 
\end{proof}  

The above proposition could also be proved by using Kobayashi's criterion~\cite[Theorem~6.3.3]{kobayashipm} and computation
of  the ``asymptotic $L_0$-support'' of~$\pi_{\gamma+\rho_{\frak{k}}}$ using Theorem~\ref{schmid}.

We conclude this section with the following remarks.

\begin{remark}{\em 
Let $G_0, K_0$ be as in \S 2.  
Recall from \S4 that one has an associated 
holomorphic discrete series $\pi_{\gamma+\rho_{\frak{k}}}$ of $K_0^*=K_1^*.K_2$.  Writing $\gamma=\lambda+\kappa$ where $\lambda,\kappa$ are dominant weights of 
$\frak{l}_1^\bc,\frak{l}_2^\bc$ respectively, 
we have $(\pi_{\gamma+\rho_{\frak{k}}})_{L_0}
=E_{\kappa}\otimes (\pi_{\lambda+\rho_{\frak{k}_1^\bc}})_{L_1}$.  Therefore $\pi_{\gamma+\rho_{\frak{k}}}$ is $L_0'$-admissible if and only if $\pi_{\lambda+\rho_{\frak{k}_1^\bc}}$ is $L_1'$-admissible. 
%=E_{\kappa}\otimes (E_{\lambda}\otimes S^*(\frak{u}_{-2}))
%=E_\gamma \otimes S^*(\frak{u}_{-2})=\oplus_{m\geq 0} E_\gamma\otimes S^m(\frak{u}_{-2}).$ 
Since $K_1$ is {\it simple}, and since $w_\frak{k}^0(\epsilon)=w_{\frak{k}_1^\bc}^0(\epsilon)$, 
it follows from the above proposition that {\it $\pi_{\gamma+\rho_{\frak{k}}}$ is $L_0'$ admissible if and only 
if $w_{\frak{k}}^0(\epsilon)\neq -\epsilon$.}
}
\end{remark}

\begin{remark}\label{waction}
{\em
Let $\Gamma$ be the set of strongly orthogonal 
roots as in Proposition \ref{sumgamma} and suppose that 
$w_\frak{k}^0(\epsilon)=-\epsilon$. Then:\\
(i) It follows 
from the explicit description of $\Gamma$ in each 
case that $w^0_{\frak{l}}(\gamma_j)=\gamma_{r+1-j}=-w_Y(\gamma_j), 1\leq j\leq r$.  In particular $-\mu\in \Gamma.$ \\
(ii)  For any $w$ in the Weyl group of $(\frak{l},\frak{t})$,  $\sum_{\gamma\in \Gamma}w(\gamma)=w(\sum_{\gamma\in \Gamma} \gamma)=-2w(\epsilon^*)=-2\epsilon^*$.\\ 
(iii)  Note that $||\gamma_i||=||\epsilon||, 1\leq i\leq r$.  
This property holds even without the assumption that 
$w^0_\frak{k}(\epsilon)=-\epsilon.$ 
}
\end{remark}

\section{Proof of Theorem \ref{nonquaternionic}}

As in \S2, let $(G_0,K_0)$ be a Riemannian symmetric pair which is not Hermitian symmetric and let $\Delta^+$ 
be a Borel-de Siebenthal root order. Let $(K_0,L_0)$ be the Hermitian symmetric pair where $\Delta_0^+$ is the positive root system of $L_0$ and 
$\Delta_0^+\cup \Delta_2$ that of $K_0$.  Recall that $\Psi_{\frak{k}}=\Psi\setminus \{\nu\}\cup \{\epsilon\}$ 
and $\Psi_{\frak{l}}=\Psi\setminus\{\nu\}$ are the 
set of simple roots of $K_0$ and $L_0$ respectively. 
The non-compact simple root of $K_0^*$ is $\epsilon$. If $w_{\frak{k}}^0(\epsilon)=-{\epsilon}$, then $w_{\frak{k}}^0(\Delta_0^+)=\Delta_0^- , w_{\frak{k}}^0(\Delta_2)=\Delta_{-2}$ and $w_Y(\Delta_0^+)=\Delta_0^+ , w_Y(\Delta_2)=\Delta_{-2}$, where $w_Y=w_{\frak{k}}^0 w_{\frak{l}}^0$. Hence $w_Y^2(\Delta_0^+\cup \Delta_2)=\Delta_0^+\cup \Delta_2$. This implies $w_Y^2=Id$. Also $w_{\frak{k}}^0(\epsilon)=-{\epsilon}$ implies $w_Y(\epsilon^*)=-\epsilon^*$.
Let $\Gamma=\{\gamma_1,\ldots,\gamma_r\} \subset \Delta_{-2}$ be the maximal set of strongly orthogonal roots obtained as in \S\ref{stronglyorthogonalroots}.

We begin by establishing the following lemma which 
will be needed in the proof of Theorem \ref{nonquaternionic}.  
We shall use the Littelmann's path model \cite{littelmann}, \cite{littelmann95}. {\it Up to the end of proof of Lemma \ref{gammasubreps} we shall use the symbols $\pi, \pi_\lambda$, etc.,   to denote LS-paths in the sense of Littelmann and are 
not to be confused with discrete series.}

Let $\lambda$ be a dominant integral weight of $\frak{k}$.  Denote by $\pi_{\lambda}$ the LS-path $t\mapsto t\lambda$, $0\leq t\leq 1,$ and by $\mathcal{P}_\lambda$ 
the set of all LS-paths of shape $\lambda$.  Recall that 
the weight of a path $\pi\in \mathcal{P}_\lambda$ is 
its end point $\pi(1)$.  Note that  $w(\pi_{\lambda})=\pi_{w(\lambda)}\in \mathcal{P}_\lambda$ for any element $w$ in the Weyl group of $\frak{k}$.
One has the Littelmann's path operator $f_\alpha, e_\alpha$, for $\alpha\in \Psi_{\frak{k}}, $ having the following 
properties which are relevant for our purposes: \\
$\bullet$ Any $\sigma\in \mathcal{P}_\lambda$ is of the 
form $\sigma=f_I(\pi_\lambda)$ for some monomial 
$f_I=f_{\beta_1}\circ \cdots \circ 
f_{\beta_k}$
in the root operators where $\beta_1,\cdots, \beta_k$ is a sequence of simple roots. (The path $\pi_\lambda$ itself corresponds to the empty sequence.) In particular, this holds for $\sigma=w(\pi_\lambda)=\pi_{w(\lambda)}$ for any $w$ in the Weyl group of $\frak{k}$.\\
$\bullet$  Let $\sigma\in \mathcal{P}_\lambda$. Then $f_\alpha(\sigma)$ (resp. $e_\alpha(\sigma)$) is either zero or belongs to $\mathcal{P}_\lambda$ and has weight $\sigma(1)-\alpha$ (resp. $\sigma(1)+\alpha$).  \\
$\bullet$  If $\pi_1*\pi_2$ is the concatenation of the paths $\pi_1, \pi_2$ where $\pi_j$ are of shapes $\lambda_j, j=1,2$, then
\[f_\alpha(\pi_1*\pi_2)=\left\{ \begin{array}{ll}
f_\alpha(\pi_1)*\pi_2 & \textrm{if $f_\alpha^n(\pi_1)\neq 0$ and $e_\alpha^n(\pi_2)=0$ for some $n\geq 1$,}\\
\pi_1*f_\alpha(\pi_2) & \textrm{otherwise}.
\end{array} \right . \eqno(4)\]
 See \cite[Lemma 2.7]{littelmann95}.

We denote by $V_{\lambda}$ (respectively $E_{\kappa}$), the finite dimensional irreducible representation of $\frak{k}$ (respectively $\frak{l}$) with highest weight $\lambda$ (respectively $\kappa$).  If $V$ is a $\frak{k}$-representation, we shall denote by $\res_{\frak{l}}(V)$ its restriction to $\frak{l}$.  By the Branching Rule \cite[p.331]{littelmann}, we have 
\[ \res_{\frak{l}}(V_{m\epsilon^*})=\sum_{\sigma} E_{\sigma(1)}\eqno(5)\]
where the sum is over all LS-paths $\sigma$ of shape $m\epsilon^*$ which are $\frak{l}$-dominant.

%\begin{remark} \label{wimagelspath}
%{\em
%If $\lambda$ is a dominant integral weight, then 

%}
%\end{remark}

\noindent
\begin{lemma} \label{branchingrule}
(i) The restriction $\res_{\frak{l}}(V_{m\epsilon^*})$ to $\frak{l}$ of the irreducible $\frak{k}$-representation $V_{m\epsilon^*}$ contains $\res_{\frak{l}}(V_{(m-p)\epsilon^*})\otimes \bc_{p\epsilon^*}$ for $0\leq p\leq m$.\\
(ii) Suppose that $w^0_{\frak{k}}(\Delta_0)=\Delta_0$. 
Then $\res_{\frak{l}}(V_{m\epsilon^*})$ contains  $\res_{\frak{l}}(V_{(m-p)\epsilon^*})\otimes \bc_{-p\epsilon^*}$. 

\end{lemma}
\begin{proof} (i)
 Note that $\pi_{m\epsilon^*}$ equals the concatenation $\pi_{(m-p)\epsilon^*}*\pi_{p\epsilon^*}$.

Let $\tau$ be an LS-path of shape $(m-p)\epsilon^*$ which
is $\frak{l}$-dominant. Then $\tau=f_{\alpha_q}\cdots f_{\alpha_1}\pi_{(m-p)\epsilon^*}$ for some sequence 
$\alpha_1,\ldots, \alpha_q$ of simple roots in $\Psi_{\frak{k}}$. 
Then $f_{\alpha_i}\ldots f_{\alpha_1}(\pi_{(m-p)\epsilon^*})\neq 0$ for $1\leq i\leq q$. It follows that 
$f_{\alpha_q}\ldots f_{\alpha_1}(\pi_{m\epsilon^*})
=f_{\alpha_q}\ldots f_{\alpha_1}(\pi_{(m-p)\epsilon^*}*\pi_{p\epsilon^*})=f_{\alpha_q}\ldots f_{\alpha_1}(\pi_{(m-p)\epsilon^*})*\pi_{p\epsilon^*}
=\tau*\pi_{p\epsilon^*}$ since $e_\alpha(\pi_{p\epsilon^*})=0$. Thus we see that if $\tau$ is any 
$\frak{l}$-dominant LS-path of shape $(m-p)\epsilon^*$, then 
$\tau*\pi_{p\epsilon^*}$ is an LS-path of shape $m\epsilon^*$. It is clear that $\tau*\pi_{p\epsilon^*}$ is $\frak{l}$-dominant. 
Since $E_{\tau*\pi_{p\epsilon^*}(1)}=E_{\tau(1)+p\epsilon^*} \cong E_{\tau(1)}\otimes \bc_{p\epsilon^*}$ and since for any path 
$\sigma$,  $\sigma*\pi_{p\epsilon^*}=\tau*\pi_{p\epsilon^*}$ implies $\sigma=\tau$, it follows that  
$\res_{\frak{l}}(V_{m\epsilon^*})$ contains $\res_{\frak{l}}(V_{(m-p)\epsilon^*})\otimes \bc_{p\epsilon^*}$ in view of (5). 

(ii) Suppose that $w^0_{\frak{k}}(\Delta_0)=\Delta_0.$ This is equivalent to the condition that $w^0_{\frak{k}}(\epsilon^*)=-\epsilon^*$, which in turn is equivalent to the requirement 
that $V_{q\epsilon^*}$ is self-dual as a $\frak{k}$-representation for all $q\geq 1$.  
Since $\res_{\frak{l}}(V_{(m-p)\epsilon^*})\otimes \bc_{p\epsilon^*}$ is contained in $V_{m\epsilon^*}$, so is its dual. That is, $\res_{\frak{l}}(V_{(m-p)\epsilon^*})\otimes \bc_{-p\epsilon^*}$ is contained in 
$\res_{\frak{l}}(V_{m\epsilon^*})$.   
\end{proof}

Although the following lemma can be deduced from the explicit branching rule in~\cite{littspherical},  at least in the case
$w^0_{\frak{k}}(\Delta_0)=\Delta_0$,   our proof below is more direct and self-contained.
\begin{lemma}\label{subreps}
 Let $0\leq p_r\leq\cdots\leq p_1\leq p_0\leq m$ be a sequence of integers. Then $\res_{\frak{l}} V_{m\epsilon^*}$ contains $E_\kappa$ where $\kappa =m\epsilon^*+p_1\gamma_1+\cdots p_r\gamma_r$.  Moreover, if 
$w^0_{\frak{k}}(\Delta_0)=\Delta_0$, then $E_\lambda$ 
occurs in $\res_{\frak{l}} V_{m\epsilon^*}$ where $
\lambda=
   (m-2p_0)\epsilon^*-(\sum_{1\leq j\leq r}p_j\gamma_{r+1-j})$.
\end{lemma}

\begin{proof}  Recall that $\gamma_1=-\epsilon$. 
Since the $\gamma_i$ are pairwise orthogonal we 
see that $s_{\gamma_i}s_{\gamma_j}=s_{\gamma_j}s_{\gamma_i}$.
Also since $\gamma_j\in \Delta_{-2}$, $\langle\epsilon^*,\gamma_i\rangle=\langle \epsilon^*,-\epsilon\rangle=-||\epsilon||^2/2$.  As noted in Remark \ref{waction}(iii), all the 
$\gamma_i$ have the same length: 
$||\gamma_i||=||\epsilon||$.  Using these facts a straightforward computation yields that 
$s_{\gamma_i}(\epsilon^*)=\epsilon^*+\gamma_i, s_{\gamma_i}(\gamma_j)=\gamma_j$ for 
$1\leq i,j\leq r, i\ne j$. Defining $p_{r+1}=0$,  it follows that 
$s_{\gamma_1}.\ldots.s_{\gamma_j}(\pi_{(p_j-p_{j+1})\epsilon^*})=:\pi_j$ is the straight-line path of weight $(p_j-p_{j+1})(\epsilon^*+\gamma_1+\cdots+\gamma_j)$ and hence we have $f_{I_j}(\pi_{(p_j-p_{j+1})\epsilon^*})=\pi_j$ for a suitable monomial in root operators $f_{I_j}$ of simple roots of $\frak{k}$ for all $2\leq j\leq r$. 
So, writing $\pi_{m\epsilon^*}=\pi_{p_r\epsilon^*}*\pi_{(p_{r-1}-p_r)\epsilon^*}*\cdots *\pi_{(p_2-p_3)\epsilon^*}*\pi_{(m-p_2)\epsilon^*}$ we have $f_{I_r}(\pi_{m\epsilon^*})=\pi_r*\pi_{(p_{r-1}-p_r)\epsilon^*}*\cdots*\pi_{(p_2-p_3)\epsilon^*}*\pi_{(m-p_2)\epsilon^*}$, in view of (4).  Clearly $f_\epsilon(\pi_j)=0$ for all $2\leq j\leq r$.  Also in view of the Proposition \ref{sumgamma}(ii), if the coefficient of a compact simple root $\alpha$ of $\frak{k}$ in the expression of $\sum_{1\leq i\leq j}\gamma_i$ is non zero, then $f_\alpha(\pi_j)=0$.  Now for a simple root $\alpha$ of $\frak{k}$, if $f_\alpha$ is involved in the expression of $f_{I_j}$, then the coefficient of $\alpha$ in the expression of $\sum_{1\leq i\leq (j+1)}\gamma_i$ is non zero. Hence $f_\alpha(\pi_{j+1})=0$ for $2\leq j\leq r-1$. Therefore $f_{I_2}\ldots f_{I_r}(\pi_{m\epsilon^*})=\pi_r*\pi_{r-1}*\cdots*\pi_2*\pi_{(m-p_2)\epsilon^*}$, in view of (4). Since $f_\epsilon(\pi_j)=0$ for all $2\leq j\leq r$ and $f_\epsilon^{p_1-p_2}(\pi_{(m-p_2)\epsilon^*})
=\pi_{(p_1-p_2)(\epsilon^*-\epsilon)}*\pi_{(m-p_1)\epsilon^*}$, we obtain $\tau:=f_\epsilon^{p_1-p_2} f_{I_2}\ldots f_{I_r}(\pi_{m\epsilon^*})=\pi_r*\cdots *\pi_2 *\pi_{(p_1-p_2)(\epsilon^*-\epsilon)}*\pi_{(m-p_1)\epsilon^*}$, again by (4). 
The break-points and the terminal point of $\tau$ are 
$p_r(\epsilon^*+\gamma_1+\cdots+\gamma_r), 
p_{r-1}(\epsilon^*+\gamma_1+\cdots+\gamma_{r-1})+p_r\gamma_r, p_{r-2}(\epsilon^*+\gamma_1+\cdots+\gamma_{r-2})+p_{r-1}\gamma_{r-1}+p_r\gamma_r,\ldots ,p_2(\epsilon^*+\gamma_1+\gamma_2)+p_3\gamma_3+\cdots +p_r\gamma_r, p_1(\epsilon^*+\gamma_1)+p_2\gamma_2+\cdots +p_r\gamma_r$ and $m\epsilon^*+p_1\gamma_1+p_2\gamma_2+\cdots+p_r\gamma_r$. 
All these are $\frak{l}$-dominant weights (since $p_1\geq p_2\geq \cdots \geq p_r\geq 0$) and so we conclude that $\tau$
is an $\frak{l}$-dominant LS-path. Hence by the branching rule, $E_{m\epsilon^*+p_1\gamma_1+p_2\gamma_2+\cdots+p_r\gamma_r}$ occurs in $V_{m\epsilon^*}$. 

Now suppose $w^0_{\frak{k}}(\Delta_0)=\Delta_0$. 
By Lemma \ref{branchingrule}, we have $\res_{\frak{l}} V_{m\epsilon^*}$ contains $\res_{\frak{l}} V_{p_0\epsilon^*}\otimes E_{(m-p_0)\epsilon^*}$.  By what has been proved already $\res_{\frak{l}}V_{p_0\epsilon^*}$ 
contains $E_{p_0\epsilon^*+p_1\gamma_1+p_2\gamma_2+\cdots+p_r\gamma_r}=:E$.  Since $V_{p_0\epsilon^*}$ is 
self-dual, $\hom(E,\mathbb{C})$ is contained in $\res_{\frak{l}}V_{p_0\epsilon^*}$.  The highest weight of $\hom(E,\mathbb{C})$ is $-p_0\epsilon^*-\sum_{1\leq j\leq r}p_jw_{\frak{l}}^0(\gamma_j)=-p_0\epsilon^*-p_1\gamma_r-p_2\gamma_{r-1}+\cdots -p_r\gamma_1$ using Remark \ref{waction}(i).    Tensoring with $E_{(m-p_0)\epsilon^*}$ we conclude that $E_\lambda$ occurs in $\res_{\frak{l}}V_{m\epsilon^*}$ with $\lambda=(m-2p_0)\epsilon^*-p_r\gamma_1 -p_{r-1}\gamma_2 -\cdots -p_2\gamma_{r-1} -p_1\gamma_r$. \end{proof}

Write $\gamma=\varphi+t\epsilon^*$ with $\langle \varphi,\mu\rangle=0$. Then $\varphi$ is $\frak{k}$-integral weight and $t$ is an integer ($\gamma$ being a $\frak{k}$-integral weight).  Also $\gamma$ is $\frak{l}$-dominant implies that $\varphi$ is $\frak{l}$-dominant. Since $\langle \gamma+\rho_{\frak{k}}, \mu \rangle <0$, we have $t<-2 \langle \rho_{\frak{k}}, \mu \rangle / ||\epsilon||^2$. Assuming $w_{\frak{k}}^0(\epsilon)=-\epsilon$, we get $\langle w_Y(\varphi), \alpha\rangle \geq 0$ when $\alpha$ is in $\Delta_0^+$ and $\langle w_Y(\varphi),\epsilon\rangle=0$. So $w_Y(\varphi)$ is $\frak{k}$-dominant integral weight.

\begin{lemma}\label{gammasubreps} With the above notation,  suppose that $w_{\frak{k}}^0(\epsilon)=-\epsilon$ and 
that $E_\tau$ is a subrepresentation of $\res_{\frak{l}}(V_{m\epsilon^*})$. Then $E_{\varphi+w_Y(\tau)}$ is a subrepresentation of 
$\res_{\frak{l}}(V_{w_Y(\varphi)+m\epsilon^*})$.
\end{lemma}

\begin{proof} Let $\pi$ denote the path $\pi_{m\epsilon^*}*\pi_{w_Y(\varphi)}$. Then $Im(\pi)$  is contained in the dominant Weyl chamber (of $\frak{k}$) and $\pi(1)=w_Y(\varphi)+m\epsilon^*$.  Since $E_\tau$ is contained in $\res_{\frak{l}}(V_{m\epsilon^*})$, there exist a sequence $\alpha_1,\ldots ,\alpha_k$ of simple roots of $\frak{k}$ such that  $f_{\alpha_1}\ldots f_{\alpha_k}(\pi_{m\epsilon^*})=:\eta$ is $\frak{l}$-dominant path with $\eta (1)=\tau$.  Since $\pi_{w_Y(\varphi)}$ is $\frak{k}$-dominant path, $\theta:=f_{\alpha_1}\ldots f_{\alpha_k}(\pi)=\eta*\pi_{w_Y(\varphi)}$, in view of (4).  Clearly $\theta$ is $\frak{l}$-dominant and $\theta(1)=\tau+w_Y(\varphi)$. Hence by the branching rule \cite[p.501]{littelmann95}, $E_{w_Y(\varphi)+\tau}$ occurs in $\res_{\frak{l}}(V_{w_Y(\varphi)+m\epsilon^*})$.

 Let $\Phi :K_0\lr GL(V_{\lambda_0})$ be the representation, where $\lambda_0:=w_Y(\varphi)+m\epsilon^*$.  Then $\phi:=d\Phi :\frak{k}_0\lr End(V_{\lambda_0})$.  For $ k\in K_0$ and $X\in \frak{k}_0$, we have
\[\Phi(k^{-1})\circ \phi(X)\circ \Phi(k)=\phi(\textrm{Ad}(k^{-1})X) \eqno(6)\]
Let $v \in V_{\lambda_0}$ is a weight vector of weight $\lambda:=w_Y(\varphi)+\tau$ such that it is a highest weight vector of $E_{\lambda}$.  Now $w_Y=(\textrm{Ad}(k)|_{i\frak{t}_0})^*$ for some $ k\in N_{K_0}(T_0)$.  Then $\Phi(k)v$ is a weight vector of weight $w_Y(\lambda)$ and it is killed by all root vectors $X_{\alpha}$ ($\alpha \in \Delta_0^+$), in view of (6); since $w_Y(\Delta_0^+)=\Delta_0^+ $.  Hence $\Phi(k)v$ is a highest weight vector of an irreducible $L_0$- submodule of $\res_{\frak{l}}(V_{\lambda_0})$.  Therefore $E_{w_Y(\lambda)}=E_{\varphi+w_Y(\tau)}$ occurs in $\res_{\frak{l}}(V_{\lambda_0})$. \end{proof}

We are now ready to prove 
Theorem \ref{nonquaternionic}.

\noindent
{\it Proof of Theorem \ref{nonquaternionic}.}  
Write $\gamma=\varphi+t\epsilon^*$ where $\langle \varphi, \mu \rangle=0$.
  
  We have 
  \[(\pi_{\gamma+\rho_{\frak{k}}})_{L_0}=E_\gamma\otimes S^*(\frak{u}_{-2})=\bigoplus (E_{\gamma}\otimes E_{a_1\gamma_1 +\cdots+a_r\gamma_r})\]
where the sum is over all integers $a_1\geq \cdots \geq a_r\geq 0$. (In view of Theorem \ref{schmid}).\\  
 So $(\pi_{\gamma+\rho_{\frak{k}}})_{L_0}$ contains $E_{\gamma+a_1\gamma_1+\cdots+a_r\gamma_r}$, for all integers $a_1\geq \cdots \geq a_r\geq 0$.

 Let $k\geq 1$ be the least integer such that $S^k(\frak{u}_{-1})$ has a one-dimensional $L_0$-subrepresentation, which is necessarily of the form $E_{q\epsilon^*}$ for some $q<0$. 
 Now $(\pi_{\gamma+\rho_{\frak{g}}})_{K_0}$ contains $\oplus_{j\geq 0} H^s(Y;\mathbb{E}_{\gamma+jq\epsilon^*})$, by Theorem \ref{bds}. 
By Borel-Weil-Bott theorem, $ H^s(Y; \mathbb{E}_{\gamma+jq\epsilon^*})$ is an 
irreducible finite dimensional $K_0$-representation with highest weight $w_Y(\gamma+jq\epsilon^*+\rho_{\frak{k}})-\rho_{\frak{k}}=w_Y(\varphi)+(-t- jq-c)\epsilon^*$ since $w_{\frak{k}}^0(\epsilon^*)=-\epsilon^*$, where $\sum_{\beta\in \Delta_2}\beta=c\epsilon^*$ for some $c \in \mathbb{N}$. Define $m_j:=-t-jq-c$ for all $j\geq 0$. For $0\leq p_r\leq \cdots \leq p_1\leq m_j$, $E_{m_j\epsilon^*+p_1\gamma_1+\cdots+p_r\gamma_r}$ is a subrepresentation of $\res_{\frak{k}}(V_{m_j\epsilon^*})$, in view of Lemma \ref{subreps}. So by Lemma \ref{gammasubreps}, $E_{\varphi-m_j\epsilon^*-p_1\gamma_r-\cdots-p_r\gamma_1}$ is a subrepresentation of $\res_{\frak{k}}(V_{w_Y(\varphi)+m_j\epsilon^*})$ since $w_Y(\gamma_j)=-\gamma_{r+1-j}$, for all $1\leq j\leq r$ by Remark \ref{waction}(i). Now 
$ H^s(Y; \mathbb{E}_{\gamma+jq\epsilon^*})$ is isomorphic to $V_{w_Y(\varphi)+m_j\epsilon^*}$. So, for $0\leq p_r\leq \cdots \leq p_1\leq m_j$, $E_{\varphi-m_j\epsilon^*-p_1\gamma_r-\cdots-p_r\gamma_1}$ is an $L_0$-submodule of $ H^s(Y; \mathbb{E}_{\gamma+jq\epsilon^*})$.

 Fix $a_1\geq \cdots \geq a_r\geq 0$, where $a_1,\ldots ,a_r\in \mathbb{Z}$. In view of Remarks \ref{spin}(i) and \ref{qodd},  
%and \ref{waction}(iii), 
$q$ is odd when $c$ is odd. Let $\mathbb{N}'=\{j\in \mathbb{N}| (jq+c) \textrm{is even} \}$. There exists $j_0\in \mathbb{N}$ such that for all $j \in \mathbb{N}'$ with $j\geq j_0$, $-(jq+c)/2\geq a_1$. Define $p_{r+1-i}:=-(jq+c)/2-a_i, \ \ 1\leq i\leq r.$ %and inductively define $p_{r-i}:=(a_i-a_{i+1})+p_{r-i+1}$ for $1\leq i\leq {r-1}$. 
Then $0\leq p_r\leq \cdots \leq p_1 <m_j.$
%$p_1= 
% p_r+(p_{r-1}-p_r)+\cdots +(p_1-p_2)=
%p_r+(a_1-a_2)+\cdots+(a_{r-1}-a_r)=
%p_r+(a_1-a_r)\leq p_r+a_1=-(jq+c)/2=(m_j+t)/2<m_j$. 

Now $\sum_{1\leq i\leq r} p_i\gamma_{r+1-i} = \sum_{1\leq i\leq r} p_{r+1-i}\gamma_i = \sum_{1\leq i\leq r} (-a_i -(jq+c)/2)\gamma_i = (jq+c)\epsilon^*-\sum_{1\leq i\leq r} a_i\gamma_i$  
 %(since $p_{r+1-i}=a_{i-1}-a_i+p_{r-i+2}=a_{i-2}-a_i+p_{r-i+3}=\cdots =a_1-a_i+p_r=-a_i-(jq+c)/2$). 
 in view of Proposition \ref{sumgamma}(i), %$\sum_{1\leq i\leq r} \gamma_i =-2\epsilon^*$ 
since $w_{\frak{k}}^0(\epsilon)=-{\epsilon}$ by hypothesis. 
%Therefore $\gamma+\sum_{1\leq i\leq r} a_i\gamma_i =$ 
It follows that $\varphi - m_j\epsilon^* -\sum_{1\leq i \leq r} p_i\gamma_{r+1-i}=
%\varphi+(t+jq+c)\epsilon^*+
%\sum_{1\leq i\leq r}(a_i %+(jq+c)/2)\gamma_i=\varphi+(t+jq+c)\epsilon^*+\sum_{1\leq i\leq r} a_i\gamma_i+({jq+c}/2) \sum_{1\leq i\leq r}\gamma_i=\varphi+(t+jq+c)\epsilon^*+\sum_{1\leq i\leq r} a_i\gamma_i-(jq+c)\epsilon^*=
\gamma+\sum_{1\leq i\leq r} a_i\gamma_i$.
% by the choice of $0\leq p_r\leq \cdots \leq p_1< m_j$. 
So for all $j\in \mathbb{N}'$ with $j\geq j_0$, $E_{\gamma+a_1\gamma_1+\cdots+a_r\gamma_r}$ is an  $L_0$-submodule of $ H^s(Y; \mathbb{E}_{\gamma+jq\epsilon^*})$. That is, for all integers $a_1\geq \cdots \geq a_r\geq 0$, the $L_0$-type $E_{\gamma+a_1\gamma_1+\cdots+a_r\gamma_r}$ occurs in $\pi_{\gamma+\rho_{\frak{g}}}$ with infinite multiplicity.

In particular, if $\gamma=t\nu^*$, each $L_0$-type in $\pi_{\gamma+\rho_{\frak{k}}}$ occurs in $\pi_{\gamma+\rho_{\frak{g}}}$ with infinite multiplicity. This completes the proof. \hfill $\Box$

\section{Appendix~1:    Borel-de Siebenthal root orders.}\label{s:appendix1}
\newcommand\fund\varpi
\noindent
Fix notation as in~\S\ref{ss:bdsrorder}.
As in \cite{ow}, we shall follow Bourbaki's notation 
\cite{bourbaki} in labeling the simple roots of $\frak{g}$. 
%which determines a positive system of $\frak{g}$.   
Let $\Psi$ be the set of simple roots of a Borel-de Siebenthal
positive root system.
We point out the simple root which is non-compact for $\frak{g}_0$ and the  
compact Lie subalgebras $\frak{k}_1, \frak{l}_1, \frak{l}_2=\frak{k}_2\subset \frak{k}_0$.  We also point out, based on Proposition \ref{relativeinvariants},  
whether the algebra $\mathcal{A}:=\mathcal{A}(\frak{u}_1,L)$  of
relative invariants is $\bc$ or $\bc[f]$. In the latter case we
indicate the value of $|f|$, the degree of~$f$.
The reader is referred to \cite{ow} for a more detailed analysis.

We also indicate the non-compact dual Hermitian symmetric space
$X:=Y^*$,  where $Y=K_0/L_0$.
In the non-quaternionic cases we point out 
whether or not $w^0_{\frak{k}}(\Delta_0)=\Delta_0$ 
(equivalently $w_{\frak{k}}^0(\epsilon)=-\epsilon$): for 
a proof see Proposition \ref{sumgamma}.  

\subsection{Table for quarternionic
  type}\label{ss:t:quart}
In all these cases, $\frak{k}_1=\frak{su}(2),
\frak{l}_1=\frak{so}(2)=i\mathbb{R}\nu^*$.  
Also $Y=\mathbb{P}^1$, $X=Y^*=SU(1,1)/U(1),$ the unit disk in $\bc$.  The condition $w^0_{\frak{k}}(\epsilon)=-\epsilon$ is trivially valid. 
%In this case $Y$ admits a spin structure. 
\ \vspace{1mm} \\
%\begin{figure*}[h]%\label{ftableclassicaltype}
\begin{tabular}{ccccc}
\toprule
\addlinespace[5pt]
$\frak{g}_0$ & 
%\multirow{3}{1.5cm}{\begin{tabular}{c}\addlinespace[3pt]Type\\ of\\
%$\frak{g}$\end{tabular}}
%\textrm{type of~$\frak{g}$} 
Type of $\frak{g}$
& $\nu$ & $\frak{l}_2$ & $\mathcal{A}$ \\
\ignore{ %%%% ignoring
   &\multicolumn{3}{c}{Classical type fundamental weights}
                &\multirow{3}{1.7cm}{\begin{tabular}{c}\addlinespace[1pt]
                                Total\\ number\end{tabular}}\\
    \cmidrule(r){2-4}
     &\multicolumn{2}{c}{Quasi-minuscule}
            &\multirow{2}{2cm}{\begin{tabular}{c}\ \ Others\end{tabular}}  & \\
    \cmidrule(r){2-3}
    & Minuscule & Non-minuscule &           & \\
}%%%%%%%%%%% End of \ignore
\midrule
\addlinespace[5pt]
$\frak{g}_0=\frak{so}(4,2l-3), l>2$  &  $B_l$ %$\frak{so}(2l+1)$ 
& $\psi_2$ &
$\frak{sp}(1)\oplus \frak{so}(2l-3)$ & $\bc[f], 
|f|=4$   \\
\addlinespace[5pt]
$\frak{so}(4,1)$ & $B_2$ %$\frak{so}(5)$ 
& $\psi_2$ & $\frak{sp}(1)$ &
$\bc$ \\
\addlinespace[5pt]
$\frak{sp}(1,l-1), l>1$ &  $C_l$ %$\frak{sp}(l)$ 
& $\psi_1$ & $\frak{sp}(l-1)$ &
$\bc$\\
\addlinespace[5pt]
$\frak{so}(4, 2l-4), l>4$ & $D_l$ %$\frak{so}(2l)$ 
& $\psi_2$ & $\frak{sp}(1)\oplus
\frak{so}(2l-4)$ & $\bc[f], |f|=4$ \\
\addlinespace[5pt]
$\frak{so}(4,4)$ & $D_4$ %$\frak{so}(8)$ 
& $\psi_2$ & $\frak{sp}(1)\oplus
\frak{sp}(1)\oplus \frak{sp}(1)$ & $\bc[f], |f|=4$\\
\addlinespace[5pt]
$\frak{g}_{2; A_1,A_1}$  %the split real form of the exceptional Lie algebra of type $G_2$. 
& $G_2$ %$\frak{g}_2$ 
& $\psi_2$ & $\frak{sp}(1)$  & $\bc[f], |f|=4$\\
\addlinespace[5pt]
$\frak{f}_{4;A_1,C_3}$  %the split  real form of the exceptional Lie algebra of type $F_4$.   
& $F_4$ %$\frak{f}_4$ 
& $\psi_1$ &$\frak{sp}(3)$  &  $\bc[f], |f|=4$\\
\addlinespace[5pt]
$\frak{e}_{6;A_1,A_5,2}$ &  $E_6$ %$\frak{e}_6$ 
& $\psi_2$ & $\frak{su}(6)$ &
$\bc[f], |f|=4$ \\
\addlinespace[5pt]
$\frak{e}_{7;A_1, D_6,1}$ & $E_7$ %$\frak{e}_7$ 
&  $\psi_1$ & $\frak{so}(12)$
& $\bc[f], |f|=4$\\
\addlinespace[5pt]
$\frak{e}_{8; A_1, E_7}$ &  $E_8$ %$\frak{e}_8$ 
& $\psi_8$ & the compact form of   %\frak{l}_2^{\mathbb{C}}=\frak{e}_7. \ \ 
$\frak{e}_7$ & $\bc[f], |f|=4$\\
\addlinespace[5pt]
\bottomrule
\end{tabular}
%\caption{Table of Borel de Siebenthal root orders of quarternionic type}\label{ftabquart}
%\end{figure}

\subsection{Table for the non-quarternionic
  type}\label{ss:t:nquart}
The non-quarternionic type Borel-de
Siebenthal root orders are listed in the following table.    The condition that
$w_{\frak{k}}^0(\epsilon)=-\epsilon$ holds precisely in the following
cases (in the others it does not):
In the first case (when  $\frak{g}_0=\frak{so}(2p, 2l-2p+1)$ with
$2<p<l, l> 3$) if and only if $p$ is even;  in the second ($\frak{g}_0=\frak{so}(2l,1), l>2$)
if and only if $l$ is even;  in the third
($\frak{g}_0=\frak{sp}(p, l-p)$, $l> 2, 1<p<l$);  
in the fourth ($\frak{g}_0=\frak{so}(2l-4, 4), l> 4$) if and only if $l$ is even;
in the fifth ($\frak{g}_0=\frak{so}(2p,2l-2p), 2<p<l-2, l>5$)
if and only if $p$ is even;  in the sixth ($\frak{g}_0=\frak{f}_{4;B_4}$);
in the eighth ($\frak{g}_0=\frak{e}_{7;A_1,D_6,2}$);   and in the tenth
($\frak{g}_0=\frak{e}_{8;D_8}$).
\begin{landscape}
%\begin{figure}[ht]%\label{ftableclassicaltype}
\begin{tabular}{ccccccc}
\toprule
\addlinespace[5pt]
$\frak{g}_0$ & 
%\multirow{3}{1.5cm}{\begin{tabular}{c}\addlinespace[3pt]Type\\ of\\
%$\frak{g}$\end{tabular}}
%\textrm{type of~$\frak{g}$} 
%$\frak{g}$& 
$\nu$ & $\frak{k}_1$ & $\frak{l}_1$ & %$\frak{l}_2$ & 
$Y$ & $X$ &
%\begin{tabular}{c} If and when \\
%  $w_{\frak{k}}^0(\epsilon)=-\epsilon$\end{tabular} 
\begin{tabular}{c} $\mathcal{A}=\bc[f]$ as indicated \\
otherwise it is $\bc$\end{tabular}\\
\ignore{ %%%% ignoring
   &\multicolumn{3}{c}{Classical type fundamental weights}
                &\multirow{3}{1.7cm}{\begin{tabular}{c}\addlinespace[1pt]
                                Total\\ number\end{tabular}}\\
    \cmidrule(r){2-4}
     &\multicolumn{2}{c}{Quasi-minuscule}
            &\multirow{2}{2cm}{\begin{tabular}{c}\ \ Others\end{tabular}}  & \\

    \cmidrule(r){2-3}
    & Minuscule & Non-minuscule &           & \\
}%%%%%%%%%%% End of \ignore
\midrule
\addlinespace[5pt]
\begin{tabular}{c}$\frak{so}(2p, 2l-2p+1)$ \\ $2<p<l$, $l>
  3$\end{tabular} & %$\frak{so}(2l+1)$ &  
$\psi_p$ &
$\frak{so}(2p)$ & $\frak{u}(p)$ &  %$\frak{s0}(2l-2p+1)$ & 
$\frac{SO(2p)}{U(p)}$
& $\frac{SO^*(2p)}{U(p)}$ &  %$\Leftrightarrow$ $p$ is even & 
$|f|=2p$ for $3p\leq2l+1$\\
%\begin{tabular}{c}$\bc[f]$ (with $\deg(f)=2p$) \\$\Leftrightarrow 3p\leq 2l+1$\end{tabular} \\
\addlinespace[5pt]
$\frak{so}(2l,1), l>2$ & %$B_l$ 
$\psi_l$ & $\frak{so}(2l)$ &
$\frak{u}(l)$  &  %& 
$\frac{SO(2l)}{U(l)}$ & $\frac{SO^*(2l)}{U(l)}$ &
%if and only if $l$ is even &  
\\ %$\bc$ \\
\addlinespace[5pt] 
\begin{tabular}{c}$\frak{sp}(p, l-p)$\\ $l> 2$, $1<p<l$\end{tabular} &
% $\frak{g}$ is of type $C_l$, 
$\psi_p$ & $\frak{sp}(p)$ &$\frak{u}(p)$ & %$\frak{sp}(l-p)$ &
$\frac{Sp(p)}{U(p)}$ &  $\frac{Sp(p,\mathbb{R})}{U(p)}$ &\ \ 
%w^0_{\frak{k}}(\epsilon)=-\epsilon$. 
\begin{tabular}{c}$|f|=p$ for $p$ even \\ such that $3p\leq 2l$\end{tabular}\\
%\begin{tabular}{c}$\bc[f]$ (with $\deg(f)=p$) \\
%$\Leftrightarrow$ $3p\leq 2l$ and $p$ even.\end{tabular}\\
\addlinespace[5pt]
\begin{tabular}{c}$\frak{so}(2l-4, 4)$ \\ $l> 4$ \end{tabular} & %$\frak{g}$ is of type $D_l$, 
$\psi_{l-2}$ & $\frak{so}(2l-4)$ & $\frak{u}(l-2)$ &
%$\frak{su}(2)\oplus \frak{su}(2)$ & 
$\frac{SO(2l-4)}{U(l-2)}$ &
$\frac{SO^*(2l-4)}{U(l-2)}$ &% \ \ w^0_{\frak{k}}(\epsilon)=-\epsilon$ 
%$\Leftrightarrow$ $l$ is even  & 
\begin{tabular}{c}  %$\bc$ if $l>6$ \\
  %$\bc[f]$, $\deg(f)=8$ if $l=6$ \\
%$\bc[f]$, $\deg(f)=6$ if $l=5$\end{tabular}
$|f|=6$ if $l=5$\\
$|f|=8$ if $l=6$
\end{tabular}\\
\begin{tabular}{c}$\frak{so}(2p,2l-2p)$ \\ $2<p<l-2$ \\
  $l>5$ \end{tabular}
%&  $\frak{so}(2l)$ 
&  $\psi_p$ &  $\frak{so}(2p)$ & 
$\frak{u}(p)$  & %$\frak{so}(2l-2p)$ & 
$\frac{SO(2p)}{U(p)}$ &
$\frac{SO^*(2p)}{U(p)}$  &     
%\ \ w^0_{\frak{k}}(\epsilon)=-\epsilon$ 
%$\Leftrightarrow$ $p$ is even &
% \begin{tabular}{c}$\bc[f]$ (with $\deg(f)=2p$)\\ if and only if
% $3p\leq 2l$\end{tabular}\\
$|f|=2p$ for $3p\leq 2l$\\
\addlinespace[5pt] 
$\frak{f}_{4;B_4}$ & %$\frak{f}_4$ &  
%having $\frak{k}_0\cong
                                %\frak{so}(9)$ as a maximal compactly
                                %embedded subalgebra. 
$\psi_4$  & $\frak{k}_0=\frak{so}(9)$  &  $i\mathbb{R}\nu^*\oplus
\frak{so}(7)$  & % & 
$\frac{SO(9)}{SO(7)\times SO(2)}$ &
$\frac{SO_0(2,7)}{SO(2)\times SO(7)}$ & %Yes & %\ \
                                %w^0_{\frak{k}}(\epsilon)=-\epsilon$.  
%\begin{tabular}{c}$\bc[f]$\\  $\deg(f)=2$ \end{tabular}\\
$|f|=2$ \\
\addlinespace[5pt]
$\frak{e}_{6;A_1,A_5,1}$  & %$\frak{e}_6$  & 
$\psi_3$  & $\frak{su}(6)$ &
$\frak{su}(5)\oplus i\mathbb{R}\nu^*$  &  %$\frak{su}(2)$  &
$\mathbb{P}^5$ &
$\frac{SU(1,5)}{S(U(1)\times U(5)}$ & %Yes & %\ \
                                %w^0_{\frak{k}}(\epsilon)\neq-\epsilon$.  
%$\bc$. 
\\
\addlinespace[5pt]
$\frak{e}_{7;A_1,D_6,2}$   &   %$\frak{e}_7$  & 
$\psi_6$   &
$\frak{so}(12)$   & 
$\frak{so}(10)\oplus i\mathbb{R}\nu^*$   & %$\frak{sp}(1)$   &
$\frac{SO(12)}{SO(2)\times SO(10)}$   &   $\frac{SO_0(2,
  10)}{(SO(2)\times SO(10)}$  & %Yes & 
%\ \ w^0_{\frak{k}}(\epsilon)=-\epsilon$ 
%$\bc$
\\
\addlinespace[5pt]
$\frak{e}_{7;A_7}$   &
%$\frak{e}_7$  & 
$\psi_2$   & $\frak{k}_0=\frak{su}(8)$   &
$\frak{su}(7)\oplus i\mathbb{R}\nu^*$   &   %&  
$\mathbb{P}^7$ & 
$\frac{SU(1,7)}{S(U(1)\times U(7)}$   &  %No &  %\ \
                                %w^0_{\frak{k}}(\epsilon)\neq-\epsilon$.   
%\begin{tabular}{c}$\bc[f]$ \\ $\deg(f)=7$ \end{tabular}\\
$|f|=7$\\
\addlinespace[5pt]
$\frak{e}_{8;D_8}$  &  %$\frak{e}_8$ & %, a real form of $\frak{e}_8$ 
$\psi_1$ & $\frak{k}_0=\frak{so}(16)$   &
$i\mathbb{R}\nu^*\oplus\frak{so}(14)$ &  %&
$\frac{SO(16)}{SO(2)\times SO(14)}$  &  $\frac{SO_0(2,14)}{SO(2)\times
  SO(14)}$  & %Yes    & %\ \ w^0_{\frak{k}}(\epsilon)=-\epsilon$.  
%\begin{tabular}{c}$\bc[f]$\\ $\deg(f)=8$\end{tabular} \\
$|f|=8$\\
%$Y$ admits a spin structure.
\bottomrule
\end{tabular}
%\caption{Table of Borel de Siebenthal root orders of non-quarternionic type}\label{ftabquart}
%\end{figure}
\end{landscape}
\section{Appendix~2:  A description of some results of Parthasarathy} 
\label{s:appendixrp}
We briefly give a description of Parthasarathy's \cite{parthasarathy} 
results on construction of certain unitarizable 
$(\frak{g},K_0)$-modules and explain how to obtain the  
description of Borel-de Siebenthal discrete series 
due to \O rsted and Wolf as Borel-de Siebenthal discrete 
series are not explicitly treated in \cite{parthasarathy}.

Let $G_0$ be a non-compact real semisimple Lie group $G_0$ with finite centre and 
let $K_0$ be a maximal compact subgroup of $G_0.$ Assume that $G_0$ contains a compact Cartan subgroup $T_0\subset K_0$.  
Let $P$ be a positive root system of $(\frak{g},\frak{t})$ and let $\frak{p}_+$ (resp. $\frak{p}_-$) equal $\sum \frak{g}_\alpha$ where the sum is over positive (respectively negative) non-compact roots. Suppose that   
$[\frak{p}_+,[\frak{p}_+,\frak{p}_+]]=0$.  
Let $B$ denote the Borel subgroup of $K=K_0^\mathbb{C}$ such that $Lie(B)=\frak{t}\oplus\sum \frak{g}_\alpha$ where the sum is over positive compact roots.    Let $P_{\frak{k}}$ and $P_n$ denote the set of compact and non-compact roots in $P$ respectively. 

Write $\rho=(1/2)\sum_{\alpha\in P}\alpha$ 
and $w_{\frak{k}}, w_{\frak{g}}$ the longest element 
of the Weyl groups of $\frak{k}$ and $\frak{g}$ with respect to 
the positive systems $P_{\frak{k}}$ and $P$ respectively. 
Let $\lambda$ be the highest weight of an irreducible representation of $K_0$ such that the following ``regularity" conditions hold: (i) $\lambda+\rho$ is dominant for $\frak{g}$, and, (ii) $H^j(K/B;\Lambda^q(\frak{p}_-)\otimes \mathbb{L}_{\lambda+2\rho})=0$ for all $0\leq j<d, 0\leq q\leq \dim\frak{p}_-$ where $d:=\dim_\bc K/B$ and $\mathbb{L}_\varpi$ denotes the holomorphic line bundle over $K/B$ associated 
to a character $\varpi$ of $T$ extended to a character of $B$ in the usual way.  From \cite[Lemma 9.1]{hp} we see that  
condition (ii) holds for $\lambda$ since $[\frak{p}_+,[\frak{p}_+,\frak{p}_+]]=0$. 
Parthasarathy shows that the $\frak{k}$-module structure on 
$\oplus _{m\geq 0}H^d(K/B;
\mathbb{L}_{\lambda+2\rho}\otimes S^m(\frak{p}_+))$ extends to a $\frak{g}$-module structure which is unitarizable. 

Suppose that $\lambda+\rho$ is regular dominant for $\frak{g}$ so that condition (i) holds. 
Then, 
the $\frak{g}$-module $\oplus_{m\geq 0} H^d(K/B;\mathbb{L}_{\lambda+2\rho}\otimes S^m(\frak{p}_+))$ is the $K_0$-finite part of a discrete series representation $\pi$ with Harish-Chandra parameter $\lambda+\rho$ and Harish-Chandra root order $P$.
The Blattner parameter is $\lambda+2\rho_n$. See \cite[p.3-4]{parthasarathy}.

Now start with a Borel-de Siebenthal positive system $\Delta^+$ where $G_0$ is further assumed to be 
simply-connected and simple. Assume also that $G_0/K_0$ is not Hermitian symmetric. The Harish-Chandra root order for the  Borel-de Siebenthal discrete series $\pi_{\gamma+\rho_\frak{g}}$ is 
$\Delta^+_0\cup\Delta_{-1}\cup \Delta_{-2}$. The 
Blattner parameter for $\pi_{\gamma+\rho_\frak{g}}$ is $\gamma+\sum_{\beta\in \Delta_2}\beta$.    
Thus, setting $P:= \Delta_0^+\cup\Delta_{-1}\cup  \Delta_{-2}$, we have $P_n=\Delta_{-1}$, $\frak{p}_+=\frak{u}_{-1}$ and 
$[\frak{p}_+,[\frak{p}_+, \frak{p}_+]]=0$ holds.

 Finally, we have the isomorphism \cite[equation (9.20)]{parthasarathy} \[H^d(K/B;\mathbb{L}_{\lambda+2\rho}\otimes 
\mathbb{S}^m(\frak{p}_+))
\cong H^s(Y;\mathbb{E}_{\lambda+2\rho_n}\otimes \mathbb{E}_{\kappa}\otimes \mathbb{S}^m(\frak{p}_+))\] of $K$-representations where $\kappa=
\sum_{\beta\in \Delta_{-2}}\beta.$ Note that $\mathbb{E}_\kappa$ is the canonical line bundle of $Y$. 
From Parthasarathy's description of the $K_0$-finite part of the discrete series $\pi_{\lambda+\rho}$ and using the above isomorphism we have  
\[\begin{array}{rcl} (\pi_{\lambda+\rho})_{K_0}&=&\bigoplus_{m\geq 0} H^d(K/B;\mathbb{L}_{\lambda+2\rho}\otimes 
\mathbb{S}^m(\frak{p}_+))\\
& \cong & \bigoplus_{m\geq 0}H^s(Y;\mathbb{E}_{\lambda+2\rho_n}\otimes \mathbb{E}_{\kappa}\otimes \mathbb{S}^m(\frak{p}_+))\\
& =& \bigoplus_{m\geq 0} H^s(Y;\mathbb{E}_{\lambda+2\rho_n+\kappa}\otimes \mathbb{S}^m(\frak{p}_+))\\
&= & \bigoplus_{m\geq 0} H^s(Y;\mathbb{E}_{\gamma}\otimes \mathbb{S}^m(\frak{u}_{-1}))\\
\end{array}\]
where $\gamma:= \lambda+2\rho_n+\kappa$. Note that 
$\gamma+\rho_\frak{g}=\lambda+2\rho_n+\kappa+
\rho_\frak{g}=\lambda+\rho$. Therefore, by \cite{ow}, the module in the last line 
is the $K_0$-finite part of $\pi_{\gamma+\rho_\frak{g}}$. Hence we see that 
Parthasarathy's description of $(\pi_{\gamma+\rho_{\frak{g}}})_{K_0}$ agrees with 
that of \O rsted and Wolf.

\end{document}